\documentclass[12pt]{article}
\usepackage{natbib}
\usepackage{tikz}
\usetikzlibrary{matrix}
\usepackage[utf8]{inputenc}
\usepackage{tabularx,xfrac}
\usepackage[text={16cm,22cm},centering]{geometry}

\usepackage{stmaryrd}

\usepackage{dingbat}

\usepackage{amssymb}   
\let\oast=\circledast

\usepackage{amsmath,amsthm}
\usepackage{verbatim}
\usepackage{enumerate}
 \usepackage{hyperref,proof,comment}
\usepackage{graphicx} 
\graphicspath{ {./images/} }

\newcommand{\DeclareRuneSeparators}[1]{} 

\input{arn.fd}
\newcommand{\ud}{\textarn{u}}

\newtheorem{theorem}{Theorem}[section]

\newtheorem{example}[theorem]{Example}

\newtheorem{lemma}[theorem]{Lemma}
\newtheorem{corollary}[theorem]{Corollary}
\newtheorem{proposition}[theorem]{Proposition}

\theoremstyle{definition}
\newtheorem{definition}[theorem]{Definition}
\newtheorem*{remark}{Remark}

\newcommand{\mM}{M}

\newcommand{\DCH}{\mathcal{H}_*}
\newcommand{\SCH}{\mathcal{H}^*}
\newcommand{\CH}{\mathcal{H}}

\newcommand{\cA}{\mathcal{A}}
\newcommand{\duA}{\lnot\cA}
\newcommand{\cB}{\mathcal{B}}
\newcommand{\cC}{\mathcal{C}}
\newcommand{\cD}{\mathcal{D}}
\newcommand{\cE}{\mathcal{E}}
\newcommand{\cF}{\mathcal{F}}
\newcommand{\cG}{\mathcal{G}}
\newcommand{\cH}{\mathcal{H}}
\newcommand{\cI}{\mathcal{I}}
\newcommand{\cJ}{\mathcal{J}}
\newcommand{\cK}{\mathcal{K}}
\newcommand{\cL}{\mathcal{L}}
\newcommand{\cP}{\mathcal{P}}
\newcommand{\cR}{\mathcal{R}}
\newcommand{\cS}{\mathcal{S}}

\newcommand{\cX}{\mathcal{X}}
\newcommand{\cY}{\mathcal{Y}}

\newcommand{\pot}{\mathcal{P}}

\newcommand*{\EE}{\mathbb{E}}
\newcommand*{\FF}{\mathbb{F}}
\newcommand*{\NN}{\mathbb{N}}
\newcommand*{\OO}{\mathbb{O}}
\newcommand*{\ZZ}{\mathbb{Z}}
\newcommand*{\LE}{\mathbb{LE}}
\newcommand*{\LF}{\mathbb{LF}}

\def\vx{\vec{x}}
\def\va{\vec{a}}
\def\vb{\vec{b}}
\def\vc{\vec{c}}
\def\vy{\vec{y}}
\def\vq{\vec{q}}
\def\vr{\vec{r}}
\def\vp{\vec{p}}
\def\vz{\vec{z}}
\def\vu{\vec{u}}

\def\vw{\vec{w}}
\def\vt{\vec{t}}

\DeclareMathOperator{\id}{id}
\DeclareMathOperator{\rg}{rg}
\DeclareMathOperator{\DD}{D}
\DeclareMathOperator{\DDd}{D^d}
\DeclareMathOperator{\CD}{CD}
\DeclareMathOperator{\gdim}{\mathfrak{Dim}}
\DeclareMathOperator{\Dim}{Dim}
\DeclareMathOperator{\Dimd}{Dim^d}
\DeclareMathOperator{\CDim}{CDim}

\DeclareMathOperator{\Card}{Card}
\let\varjo=\partial
\DeclareMathOperator{\Crit}{Crit}
\DeclareMathOperator{\Critd}{Crit^d}
\DeclareMathOperator{\Max}{Max}
\DeclareMathOperator{\Min}{Min}
\newcommand*{\anonym}{\mathrel{\Upsilon}}

\DeclareMathOperator{\FO}{FO}

\newcommand{\ilor}{{\underline{\lor}}}
\newcommand{\tland}{{\owedge}}
\newcommand{\gtd}{\bigvee\limits}
\DeclareMathOperator{\dom}{dom}
\DeclareMathOperator{\len}{len}
\let\osaj=\subseteq
\let\jer=\smallsetminus
\let\tyj=\emptyset
\let\aioj=\subsetneq

\DeclareMathOperator{\dep}{\,=}
\def\ivee{\underline{\vee}}
\newcommand*{\sat}[2]{\left\Vert #1 \right\Vert^{#2}}

\DeclareMathOperator{\nem}{\mathsf{N}\hspace{-0.25em}\mathsf{E}}
\def\ev{\mbox{\tt\bf E}}
\title{Dimension in team semantics}
\begin{document}

\author{Lauri Hella\\University of Tampere,\\
         Tampere, Finland
        
\and Kerkko Luosto\\
University of Tampere,\\
         Tampere, Finland
    
\and
Jouko V\"a\"an\"anen\footnote{Supported by  the Academy of Finland (grant No 322795) and the European Research Council (ERC) under the
European Union’s Horizon 2020 research and innovation programme (grant No
101020762).}\\
University of Helsinki,\\
         Helsinki, Finland
        }

\maketitle

\begin{abstract}
We introduce three measures of complexity for families of sets.
Each of the three measures, that we call dimensions, is defined in
terms of the minimal number of convex subfamilies that are needed
for covering the given family: for upper dimension, the subfamilies
are required to contain a unique maximal set, for dual upper
dimension a unique minimal set, and for cylindrical dimension both a
unique maximal and a unique minimal set. In addition to considering
dimensions of particular families of sets we study the behaviour of
dimensions under operators that map families of sets to new families
of sets. We identify natural sufficient criteria for such operators
to preserve the growth class of the dimensions.

We apply the theory of our dimensions for proving new hierarchy
results for logics with team semantics. First, we show that the
standard logical operators preserve the growth classes of the families
arising from the semantics of formulas in such logics. Second, we show
that the upper dimension of $k+1$-ary dependence, inclusion,
independence, anonymity, and exclusion atoms is in a strictly higher
growth class than that of any $k$-ary atoms, whence the $k+1$-ary
atoms are not definable in terms of any atoms of smaller arity.

\end{abstract}

\section{Introduction}

Families of sets are well-studied in discrete mathematics and set
theory (see e.g. \cite{B}). Sperner families and downward
closed families are examples of basic building blocks that can be used
to analyse complex families. Considerations on ways how to represent a
family as a union of more basic families leads us to several concepts
of dimension. Given the finite size of the base set, we use our
dimensions to associate families of subsets of the base set with
better quantitative estimates than their mere size.  We show that
certain canonical operations on families of sets preserve
dimension. This allows us to isolate dimension bounded collections of
families of sets.

By restricting attention to families of subsets of cartesian powers of
finite sets we obtain finer distinctions. Such families arise in the
context of so-called team semantics. In ordinary Tarski semantics of
first order logic $\FO$ and its extensions by new logical operations any
formula and any model of the appropriate kind can be associated with
the set of assignments satisfying the formula in the model. It is
natural to consider such a set as a subset of the cartesian power of
the domain of the model. In team semantics satisfaction is defined
with respect to sets (`teams') of assignments. Accordingly, any
formula becomes associated with a family of subsets of such a
cartesian power. We use our dimensions and preservation results for
logical operations to prove new non-definability and hierarchy results
for logics based on teams semantics. Examples of such logics are
dependence logic, independence logic and inclusion logic.

The background of our work for this paper is the following. Ciardelli
defined in his Master's Thesis\footnote{\cite{ciardelli09}} a dimension concept,
in the case of downward closed families, namely the cardinality of the
set of maximal sets, or equivalently, the smallest number of
power-sets that cover the family. He proved the preservation
properties for basic propositional logic operations, including
intuitionistic implication, and referred to them as Groenendijk
inequalities.  In \cite{HLSV} 
a similar dimension concept was introduced in modal logic, including preservation of
dimension results for logical operations of modal dependence
logic. Hella and Stumpf used a form of dimension to prove a
succinctness result for the inclusion atom in modal inclusion logic (\cite{HS}).
In \cite{LVil}  the notion of dimension was generalized
from downward closed families to arbitrary families. They proved
preservation of dimension under propositional operations, and computed
the dimension of dependence and exclusion atoms in the context of
propositional logic. An important step in the background of this paper has been also \cite{Lu}. 

There are several other dimension concepts in discrete
mathematics. Perhaps the most famous is the matroid rank, which
coincides with the usual concept of dimension in the case of vector
spaces and with degree of transcendence in the case of algebraically
closed fields. However, our families do not necessarily satisfy the
Exchange Axiom of matroids and therefore this concept does not work in
our context. Another well-known dimension is the Vapnik–Chervonenkis-
or VC-dimension. In Section~\ref{vc} we argue that VC-dimension is not
preserved by logical operations in the sense that our dimension
is. Therefore it does not serve our purpose well in this paper. Still
another dimension is the length of a disjunctive normal form in
propositional logic. We show in Section~\ref{dnf} that this is
equivalent to one of the dimensions (cylindrical dimension) we
investigate.



The concepts we introduce in this paper belong to discrete mathematics
with no immediate connection to logic. Thus part of this paper can be
read with no knowledge or interest in logic. However, our applications
come from logic, more exactly from team semantics. We believe that our
results are an interesting new contribution to discrete mathematics of
families of sets. At the same time, we suggest that our results lead
to a new approach to definability questions in team semantics and, in
particular, yield a new strong hierarchy result
(Theorem~\ref{hierarchy}).

\subsection*{An outline of the paper}

Section 2 gives the basic concepts of out dimension theory. We define three dimension notions for arbitrary families of sets and give some elementary basic properties of these notions. We define the basic operators on families of sets that we will use in our results. Finally, we introduce some concepts from logic that are relevant for our results. In particular, we introduce the so-called team semantics which gives rise to a wealth of interesting families of subsets of Cartesian products $M^k$ of finite sets, raising the question what the dimensions of these families are.

Section 3 introduces some technical tools for explicit dimension computations. Such computations are the heart of our results.

In Section 4 we introduce the concept of a growth class. These classes are used to measure the rate of growth of dimension of definable sets of subsets of a Cartesian product $M^k$ when the finite size of $M$ increases. Some important  results are proved about preservation of  dimension under operators. These preservation results will make it easier to estimate the growth class of a given definable family of sets.

In Section 5 we put our results together and indicate applications. Our main application is Theorem~\ref{hierarchy} which gives  strong hierarchy results for a number of logics based on team semantics. We also observe that several logical operations that occur in the literature of team semantics are not of the kind that preserve dimension. This allows us to use the quantitative method of dimension to obtain qualitative distinctions between logical operations.

In Section 6 we address the obvious question why not apply the VC-dimension. The answer turns out to be that VC-dimension is not preserved under the logical operations that we are mainly interested in, such as conjunction, disjunction, existential quantifier and universal quantifier.

In Section 7 we relate one of our dimension concepts to an invariant related to disjunctive normal forms of Boolean polynomials. This allows us to make some conclusions about dimensions of random families of sets.

Finally, in Section 8 we show that it is impossible to obtain on infinite domains the kind of results we are after. The desired hierarchy results are simply false on infinite domains.

\section{Basic notions }



\subsection{Families of sets}

In the sequel, our applications will build on heavy use of
combinatorial results in the subfield often called set-system
combinatorics.  We start with commonly used notions.

We use standard set-theoretic notation, including the shorthands
\[
  \bigcup\cA = \bigcup_{A\in\cA}A \text{ and }
  \bigcap\cA = \bigcap_{A\in\cA}A,
\]
the latter being unambiguous only if $\cA\neq\emptyset$.
In addition, we write
\[
  [A,B] = \{ C \mid A\osaj C\osaj B \}, 
\]
for any sets $A$ and $B$.  Note that if $A\not\osaj B$, then $[A,B]=\tyj$.

\begin{definition}
Let $\cA$ be a family of sets.  The family $\cA$ is an
\emph{interval or cylinder}, if there exist $A_0$ and $A_1$
such that $A_0\osaj A_1$ and $\cA=[A_0,A_1]$.  The family $\cA$
is \emph{convex} if for all $S,T\in\cA$, we have $[S,T]\osaj\cA$.
$\cA$ is \emph{downwards closed} if $A\in\cA$ and $S\osaj A$
imply $S\in\cA$. The family $\cA$ is a \emph{Sperner family} if
for all distinct $S,T\in\cA$ we have $S\not\osaj T$.
Finally, $\cA$ \emph{fulfills the Zorn condition} if
it is closed under nonempty unions of chains, i.e., if $\cC$ is a
nonempty chain  (or a nonempty family linearly ordered by inclusion), then
$\bigcup\cC\in\cC$.  A stricter notion is also useful:
$\cA$ is \emph{closed under unions}
if for every subfamily $\cB\osaj\cA$, we have $\bigcup\cB\in\cA$.
\end{definition}

Note that if a family of sets is downward closed or a Sperner family,
then it is also convex.  The concept of a upward closed family is also
useful in set theory but is lacking here, as applications of our methods
are very much leaning towards downward closed families.

The concept of Zorn condition is only used when we discuss the applicability
of these notions to infinite families.   Our emphasis is, however,
on finite families of finite sets, for which the Zorn condition,
as we have formulated it, trivially holds.  We have weakened the standard
condition by imposing the requirement only on nonempty chains.
The only notable effect is that the empty set is not required to be included
in a family fulfilling the Zorn condition, thus allowing all the finite
families to meet the condition.

For $\cA$ a family of sets, we denote the family of all maximal (with
respect to inclusion) sets in~$\cA$ by $\Max(\cA)$. Similarly,
$\Min(\cA)$ is the set of all minimals sets in $\cA$.  Observe that
$\Max(\cA)$ and $\Min(\cA)$ are always Sperner families.

\begin{definition} 
A family of sets $\cA$ is \emph{dominated (by $\bigcup\cA$)}
if $\bigcup\cA\in\cA$.  The family $\cA$ is \emph{supported}
(by $\bigcap\cA$) if $\cA$ is nonempty and $\bigcap\cA\in\cA$.
Naturally, we say that $\cA$ is \emph{dominated convex} if it is
dominated and convex. Similarly, $\cA$ is \emph{supported convex} if it is
supported and convex.
\end{definition}

In other words, a family $\cA$ is dominated by a set $D$ if and only
if $D$ is the largest element in $\cA$ with respect to inclusion.
Similarly,  $\cA$ is supported by a set $S$ if and only
if $S$ is the smallest element in $\cA$ with respect to inclusion.
We spell out some of the easily seen connections between the
basic concepts in the following lemma.

\begin{lemma}  \label{dom-supp-conv}
Let $\cA\osaj\pot(X)$ where $X$ is a set. Denote $\duA=\{X\jer A \mid A\in\cA \}$.
\begin{enumerate}[(a)]
\item
The family $\cA$ is an interval if and only if it is dominated, supported and convex.
\item
$\cA$ is convex if and only if $\duA$ is convex.
\item
$\cA$ is dominated if and only if $\duA$ is supported. \qed 
\end{enumerate}
\end{lemma}

We proceed to the central dimension concepts which will be studied throughout
this paper.  The upper dimension was first defined for downwards closed families
in  \cite{HLSV}
and subsequently generalized for arbitrary families  in \cite{LVil}. The definition presented here is an equivalent
reformulation of the latter. We also introduce two new dimension concepts.
The idea of the dual upper dimension is that many of the underlying
ideas behind the upper dimension work if the inclusion order is reversed,
as the previous lemma indicates.  The third concept, cylindrical dimension,
can be seen as a combination of the two mentioned dimension concepts.

\begin{definition}\label{dimension-def}
Let $\cA$ be a family of sets.  We say that a subfamily
$\cG\osaj\cA$ \emph{dominates} $\cA$ if there exist
dominated convex families $\cD_G$, $G\in\cG$,
such that $\bigcup_{G\in\cG}\cD_G=\cA$ and $\bigcup\cD_G=G$,
for each $G\in\cG$.
The subfamily $\cK\osaj\cA$ \emph{supports} $\cA$ if there exist
supported convex families $\cS_K$, $K\in\cK$ such that
$\bigcup_{K\in\cK}\cS_K=\cA$ ja $\bigcap\cS_K=K$, for each $K\in\cK$.

The \emph{upper dimension} of the family is $\cA$ 
\[
\DD(\cA)=\min\{|\cG| \mid \text{$\cG$ dominates the family $\cA$}\},
\]
\emph{the dual upper dimension} is
\[
\DDd(\cA)=\min\{|\cG| \mid \text{$\cG$ supports the family $\cA$}\}
\]
and the \emph{cylindrical dimension} is
\[
  \CD(\cA)=\min\{|I| \mid \text{$(\cA_i)_{i\in I}$ is an indexed family of
    intervals with $\bigcup_{i\in I}\cA_i=\cA$}\}.
\]

\end{definition}

\begin{proposition}\label{basic:dimEst}
Let $\cA$ be a family of set.  Then
\[
  \DD(\cA)\le\CD(\cA) \text{ and } \DDd(\cA)\le\CD(\cA).
\]
If, in addition, $\cA$ is convex, then
\[
  \CD(\cA)\le\DD(\cA)\DDd(\cA)
\]
\end{proposition}

\begin{proof}
Let $(\cA_i)_{i\in I}$ be an indexed family of minimal size of intervals 
covering $\cA$, i.e., $\bigcup_{i\in I}\cA_i=\cA$.
Write $\cA_i=[B_i,C_i]$, for each $i\in I$, and consider
the families $\cB=\{B_i \mid i\in I\}$ and $\cC=\{C_i\mid i\in I\}$.
Then $\cA_i$ is a convex set supported by $B_i$ and dominated
by $C_i$, for $i\in I$.  Consequently, $\cB$ supports $\cA$
and $\cC$ dominates $\cA$, which implies
\[
  \DDd(\cA)\le|\cB|\le|I|=\CD(\cA) \text{ and }
  \DD(\cA)\le|\cC|\le|I|=\CD(\cA).
\]

For the second part of the proposition, assume now that $\cA$
is convex.  Let $\cG$ be a family of minimal size that dominates $\cA$
and $\cK$ be a family of minimal size that supports $\cA$.
Then $\DD(\cA)=|\cG|$ and $\DDd(\cA)=|\cK|$.
Let $I$ be the set of pairs $(G,K)\in \cG\times\cK$ with $K\osaj G$.
By convexity of $\cA$, we have $[K,G]\osaj\cA$, for each $(G,K)\in I$.
On the other hand, if $A\in\cA$, then there has to be $G\in\cG$ such
that $A\osaj G$, as $\cG$ dominates $\cA$, and similarly $K\in\cK$
such that $K\osaj A$.  This means that $A\in[K,G]$, for some
interval $[K,G]$ with $(G,K)\in I$.  Consequently,
\[
  \cA=\bigcup_{(G,K)\in I}[K,G],
\]
which implies
\[
  \CD(\cA)\le|I|\le |\cG\times\cK|=\DD(\cA)\cdot\DDd(\cA).
\]
\end{proof}

Clearly, if $\cA\osaj\pot(X)$ with $n=|X|\in\NN$, then
$\CD(\cA)\le 2^n$. One gets easily a modest improvement to this
result, which is the best possible upper bound,
as the succeeding example shows.

\begin{proposition}
Let $X$ be a nonempty finite set with $n=|X|$, and let $\cA\osaj\pot(X)$.  
Then
\[
  \CD(\cA)\le 2^{n-1}.
\]
Hence also, $\DD(\cA)\le 2^{n-1}$ and
$\DDd(\cA)\le 2^{n-1}$.
  
\end{proposition}

\begin{proof}
Fix $b\in X$ and consider the partition of $\pot(X)$ in
pairs~$\{A,A\cup\{b\}\}$, where $A\osaj X\jer\{b\}$.
For each such pair, $\cA\cap\{A,A\cup\{b\}\}$ is either
empty or one of the intervals $[A,A]=\{A\}$,
$[A\cup\{b\},A\cup\{b\}]=\{A\cup\{b\}\}$ or
$[A,A\cup\{b\}]=\{A,A\cup\{b\}\}$. Consequently, there is a family of
at most $2^{n-1}$ intervals, the union of which is $\cA$. The remaining claims follow from Proposition~\ref{basic:dimEst}.
\end{proof}

\begin{example} \label{basic:evenEx}
Let $X$ be a nonempty finite set of $n$~elements.
Consider the family
\[
  \cE=\{A\osaj X \mid |A|\text{ is even} \}.
\]
Let $\cG$ be a subfamily of $\cE$ which dominates $\cE$. 
Let $A\in\cE$ and suppose that $A$ is dominated by $G\in\cG$, which means
that $A$ belongs to certain dominated convex family $\cD_G$
where $\cD_G\osaj\cE$ and $\bigcup\cD_G=G$.  By convexity,
$[A,G]\osaj\cD_G\osaj\cE$. Note that $A,G\in\cA$ both have even size. 
However, the interval $[A,G]$ would contain sets of odd size
unless $A=G$. As this holds for arbitrary $A\in\cE$, we conclude that
$\cG=\cE$.  Hence, $\DD(\cE)=|\cE|=2^{n-1}$.  By symmetry, we get
$\DDd(\cE)=2^{n-1}$, too.  Combined with the last two propositions,
we have that $\CD(\cE)=2^{n-1}$.
\end{example}

\subsection{Operators}

In addition to studying the dimensions of fixed families of sets we
are also interested in the behaviour of dimensions under various
operators. An operator on families of sets on a fixed base set $X$ is
a function $\Delta\colon \cP(\cP(X))^n\to\cP(\cP(X))$
for some positive integer $n$. 
In some applications to team semantics it is useful to consider more
general operators of the form
$\Delta\colon \cP(\cP(X))^n\to\cP(\cP(Y))$ with different base sets
$X$ and $Y$. We list in the next example some natural set-theoretic
operators that we will study further in the forthcoming sections.

\begin{example}\label{operator-ex}
Let $X$ be a base set.

\begin{enumerate}[(a)]
\item
\emph{Union and intersection.} The union operator
$\Delta^X_\cup\colon \cP(\cP(X))^2\to\cP(\cP(X))$ on the base set $X$
is defined by $\Delta^X_\cup(\cA,\cB)=\cA\cup\cB$. Similarly, the
intersection operator
$\Delta^X_\cap\colon \cP(\cP(X))^2\to\cP(\cP(X))$ on $X$ is defined by
$\Delta^X_\cap(\cA,\cB)=\cA\cap\cB$.

\item
\emph{Complementation.}  Complementation on $X$ is the unary
operator $\Delta^X_c\colon \cP(\cP(X))\to\cP(\cP(X))$ defined by
$\Delta^X_c(\cA)=\cP(X)\jer\cA$.

\item
\emph{Tensor disjunction and conjunction.} The idea of tensor 
disjunction\footnote{We call this operator tensor \emph{disjunction}, 
since it gives the team semantics for disjunction.} $\Delta^X_\lor$ and 
tensor conjunction $\Delta^X_\land$
is to take unions and intersections inside the families: 
$\Delta^X_\lor(\cA,\cB)=\{A\cup B\mid A\in\cA, B\in\cB\}$ and
$\Delta^X_\land(\cA,\cB)=\{A\cap B\mid A\in\cA, B\in\cB\}$.

\item
\emph{Tensor negation.} Pushing complementation inside a given
family, we obtain tensor negation:
$\Delta^X_\lnot(\cA)=\{X\jer A\mid A\in\cA\}$.

\item
\emph{Projections.} Let $f\colon X\to Y$ be a surjective
function. The (abstract) projection operator corresponding to $f$ is
obtained by lifting $f$ to a function
$\Delta_f\colon\cP(\cP(X))\to\cP(\cP(Y))$ in the usual way:
$\Delta_f(\cA)=\{f[A]\mid A\in \cA\}$, where $f[A]$ denotes the image
$\{f(a)\mid a\in A\}$ of $A$ under $f$.

\item
\emph{Inverse projections.} Given a surjection $f\colon X\to Y$,
we can also define a useful operator
$\Delta_{f^{-1}}\colon\cP(\cP(Y))\to\cP(\cP(X))$ as follows:
$\Delta_{f^{-1}}(\cB)=\{A\in\cP(X)\mid f[A]\in\cB\}$.

\item
\emph{Existential and universal quantification.}
Consider the concrete projection function $f\colon X\to Y$ for
$X=X_0\times\cdots\times X_{m-1}$ and
$Y=X_0\times\cdots X_{i-1}\times X_{i+1}\times\cdots\times X_{m-1}$
defined by
$f(a_0,\ldots,a_{m-1})=(a_0,\ldots,a_{i-1},a_{i+1},\ldots,a_{m-1})$
(i.e., $f$ is the projection to coordinates $j\not=i$). Note that
$B\in \Delta_f(\cA)$ if and only if there is $A\in\cA$ such that for
each tuple $\va\in B$ there exists some element $a\in X_i$ such that
the extension of $\va$ by $a$ as the $i$th component is in $A$. Thus,
$\Delta_f$ corresponds to the logical operation of existential
quantification, and accordingly we denote it by
$\Delta^X_{\exists i}$.

Similarly, we define an operator
$\Delta^X_{\forall i}\colon\cP(\cP(X))\to\cP(\cP(Y))$ that corresponds
to universal quantification: Given a set $B\in\cP(Y)$, let
$B[X_i/i]=\{(a_0,\ldots,a_{m-1})\in X\mid
(a_0,\ldots,a_{i-1},a_{i+1},\ldots,a_{m-1})\in B, a_i\in X_i\}$. Then
we let $\Delta^X_{\forall i}(\cA)=\{B\in\cP(Y)\mid B[X_i/i]\in \cA\}$.
\end{enumerate}

\end{example}

Note that the union and intersection operators $\Delta^X_\cup$ and $\Delta^X_\cap$ 
do not depend on the base set $X$.
Thus, in the sequel we will denote these operators simply by $\cup$ and $\cap$. The same holds
for tensor disjunction and conjunction, whence we will use the notation  $\cA\lor\cB:=\Delta^X_\lor(\cA,\cB)$ 
and $\cA\land\cB:=\Delta^X_\land(\cA,\cB)$. On the other hand, both complementation $\Delta^X_c$
and tensor negation $\Delta^X_\lnot$ depend on $X$, whence we do not introduce any shorthand notation 
for them.

Note further that the projections $\Delta^X_{\exists i}$ do not depend
on $X=X_0\times\cdots\times X_{m-1}$, since the length and the $i$-th
component of any tuple $\va$ is uniquely determined:
$(a_0,\ldots,a_{m-1})=(b_0,\ldots,b_{m'-1})$ if and only if $m=m'$ and
$a_i=b_i$ for all $i< m$. However, the universal projection operator
$\Delta^X_{\forall i}$ clearly depends on the base set $X$. Thus, for
the sake of uniformity we keep using the notation
$\Delta^X_{\exists i}$.

\subsection{Tensor operators}

We have seen in Example~\ref{operator-ex} that the 
disjunction and conjunction connectives give rise to  tensor disjunction
and tensor conjunction operators.  This idea can of course be generalized to
arbitrary connectives.  We introduce here the related concept of
tensor operator, and show that they preserve intervals but not
necessarily dominated convex or supported convex families.

\begin{definition}\label{tensor-def}
Fix a base set $X$ and let $\oast$ be a binary operation
on the set $\{0,1\}$, i.e., $\oast$ is a map $\{0,1\}\times\{0,1\}\to\{0,1\}$.  
Then the \emph{corresponding set-theoretic operation} is $*\colon\pot(X)\to\pot(X)$,
\[
  A*B=\{x\in X \mid \chi_A(x)\oast\chi_B(x)=1\}
\]
where $\chi_C$ is the characteristic function related to a set $C$.
The \emph{tensor operator} corresponding to $\oast$ is
$\Delta^X_{\oast}\colon \pot(\pot(X))\to\pot(\pot(X))$,
\[
  \Delta^X_{\oast}(\cA,\cB)=\{A*B \mid A\in\cA,B\in\cB\}.
\]
\end{definition}

\begin{remark}
\begin{enumerate}[(a)]
\item 
  Naturally, we often identify the binary operation $\oast$ with
the corresponding connective, especially on the notational level.
We also overload the notation, writing in the sequel simply
\[
  \cA\oast\cB=\Delta^X_{\oast}(\cA,\cB).
\]
Note though, that this notation is independent of the set~$X$
only if $0\oast 0=0$.

\item
We could have considered $n$-ary operations on $\{0,1\}$ in general,
which appears to be a non-trivial generalization, but we refrain
ourselves from doing that here.  Even so, it is worth-while
to have notation
\[
  \lnot\cA = \Delta^X_{\lnot}(\cA)=\{X\jer A \mid A\in\cA\} 
\]
for the unary operation corresponding to negation.
\end{enumerate}
\end{remark}

Note that we always have $\cA\oast\tyj=\tyj\oast\cB=\tyj$.

There are $2^4=16$ binary operations on the set $\{0,1\}$,
4 of which (constant functions and projections) are rather trivial.
Among the $8$ zero-preserving (i.e., $0\oast 0=0$) operations
there are $5$ non-trivial tensor operations, which are listed
below except for the case $\cA\oast\cB=\cB-\cA$, which is the
set difference with the roles reversed.

\def\arraystretch{1.2}
\begin{center}
\begin{tabularx}{\linewidth}{llX}
connective  & set-theoretic operation & tensor operation \\
\hline
disjunction $\lor$  & union $\cup$
    & $\cA\lor\cB  = \{ A\cup B \mid A\in\cA,\,B\in\cB \}$ \\  
conjunction $\land$ & intersection $\cap$ 
    & $\cA\land\cB = \{ A\cap B \mid A\in\cA,\,B\in\cB \}$ \\
``$p$ but not $q$''   & set difference $\jer$
    & $\cA-\cB  = \{ A\jer B \mid A\in\cA,\,B\in\cB \}$ \\
exclusive disjunction $\oplus$ & symmetric difference $\triangle$
    & $\cA\oplus\cB = \{ A\triangle B \mid A\in\cA,\,B\in\cB \}$  
\end{tabularx}
\end{center}
\def\arraystretch{1}

If the connective $\oast$ is commutative (resp.~associative),
then the corresponding tensor operation
is  commutative (resp.~associative), too,
but as we shall see in the next example,
the same does not apply to idempotence.  In general, the well-known
logical equivalences do not transfer to equalities about
tensor operations.  This means that, in contrast to propositional
logic where we often reduce problems to some small set of basic
connectives,  it is better to consider tensor operators separately.

\begin{example}
Suppose the base set $X$ is infinite and $\cA=\{\{x\} \mid x\in X\}$.
Consider the families
\[
 \cA_n=\underbrace{\cA\lor\ldots\lor\cA}_{n\text{ times}}, 
\]
for $n\in\ZZ_+$.  An easy induction shows that
$\cA_n=\{B\osaj X \mid B\neq\tyj,|B|\le n\}$, so these families
are all different.  In particular, $\cA\lor\cA=\cA_2\neq\cA$,
so the tensor operation $\lor$ is not idempotent, though
the binary operation $\lor$ on $\{0,1\}$ is.  A similar example
shows that $\land$ (as a tensor operator) is not idempotent
either.

Elaborating on this  example, one sees that distributive
law does not hold for $\lor$ and $\land$, either.
Choose $\cB=\cC=\{X\}$ with $\cA$ as above; then
\[
  \cA\land(\cB\lor\cC)=\cA\land\{X\}=\cA\neq
  \cA\lor\cA=(\cA\land\cB)\lor(\cA\land\cC).
\]
\end{example}

We do not aim at a complete analysis on how the tensor operations
behave, but we shall show that they preserve intervals.
In connection with the following lemma we will have thus one way
to compute the result of tensor operation.

\begin{lemma}
Let $\oast$ be a binary operation on $\{0,1\}$, and let
$\cA_i,\cB_j\in\pot(\pot(X))$ be families of sets, for $i\in I$ and $j\in J$.  
Then
\[
\left(\bigcup_{i\in I}\cA_i\right)\oast\left(\bigcup_{j\in J}\cB_j\right)
=\bigcup_{\substack{i\in I,\\j\in J}}(\cA_i\oast\cB_j).
\] 
\end{lemma}

\begin{proof}
The reader can either prove this as an easy exercise, or wait until
Section~\ref{growth}, where it is shown that tensor operators
are so-called Kripke operators and that the Union Lemma \ref{union-lem}
holds generally for Kripke operators.
\end{proof}

We need some auxiliary concepts to handle with intervals and
tensor operators.  We depart for a moment from classical logic
(Kleene introduced his logic in \cite[\S 64]{Kl}),
and introduce a new truth value $\ud\neq 0,1$ for 'unknown'.
\footnote{This symbol stands for the letter 'u' as written in runes.
'Unknown' is 'ukjent' in Norwegian.}

\begin{definition}
Let $\oast$ be a binary operation on $\{0,1\}$.  We define
\emph {Kleene's extension} $\widetilde\oast$ of $\oast$ as follows.
Write $V_0=\{0\}$, $V_1=\{1\}$ and $V_\ud=\{0,1\}$ and
$A\oast B=\{u\oast v \mid a\in A, b\in B\}$, for $A,B\osaj\{0,1\}$.
Then $\widetilde\oast$ is determined by the rule:
\[
  u\widetilde\oast v = w \text{ if and only if }V_u\oast V_v = V_w, 
\]
for $u,v,w\in\{0,1,\ud\}$.
Overloading once again the notation, we shall
denote also the extension by $\oast$ instead on $\widetilde\oast$
in the sequel.
\end{definition}

\begin{definition}
\begin{enumerate}[(a)]
\item 
\emph{The characteristic function} of a family of sets
$\cA\osaj\pot(X)$  is \break $\xi_\cA\colon X\to \{0,1,\ud\}$,
\[
  \xi_\cA(x)=
\begin{cases}
  1, &\text{for }E(x)=\{1\} \\
  \ud, &\text{for }E(x)=\{0,1\} \\
  0, &\text{for }E(x)=\{0\}
\end{cases}
\]
where $E(x)=\{\chi_A(x) \mid A\in\cA \}$.

\item
We say that $\chi\colon X\to\{0,1\}$ is \emph{compatible}
with the function $\xi\colon X\to\{0,\ud,1\}$
if for all $x\in X$, $\xi(x)\neq\ud$ implies $\chi(x)=\xi(x)$.
\end{enumerate}
\end{definition}

\begin{lemma} \label{basic:compat}
Let $\cA\osaj\pot(X)$.  
\begin{enumerate}[(a)]
\item
For all $A\in\cA$, we have that $\chi_A$ is compatible with $\xi_\cA$.

\item 
The family $\cA$ is an interval provided that the following
condition holds for every $A\osaj X$:
$A\in\cA$ if and only if $\chi_A$ is compatible with $\xi_\cA$.
Conversely, if $\cA$ is an interval then the condition holds.
\end{enumerate}
\end{lemma}

\begin{proof}
(a)  Let $A\in\cA$ and for every $x\in X$.
Using the notation of the previous definition, we note that
$\chi_A(x)\in E(x)$.
Thus, either $\xi_\cA(x)=\chi_A(x)$ or $\xi_\cA(x)=\ud$, and compatibility
follows.

(b) Suppose
\[
  \cA=\{ A\osaj X \mid \chi_A\text{ is compatible with }\xi_\cA \}.
\]
Put $B=\xi^{-1}_\cA[\{1\}]$ and $C=\xi^{-1}_\cA[\{1,\ud\}]$.
Then for every $A\osaj X$, compatibility of $\chi_A$ with $\xi_\cA$
is equivalent to the condition $B\osaj A\osaj C$.
Hence, $\cA=[B,C]$.  The converse direction is easy.
\end{proof}

\begin{lemma} \label{basic:comp-op}
Let $\oast$ be a binary operation on $\{0,1\}$ and
$\cA,\cB\osaj\pot(X)$.  Then for every $x\in X$,
it holds that
\[
\xi_{\cA\oast\cB}(x)=\xi_\cA(x)\oast\xi_\cB(x).
\]  
\end{lemma}

\begin{proof}
Write $E_\cC(x)=\{\chi_C(x) \mid C\in\cC\}$, for $\cC\osaj\pot(X)$ and $x\in X$.
Then  
\begin{align*}
E_{\cA\oast\cB}(x)
  & = \{ \chi_C(x) \mid C\in \cA\oast\cB \} \\      
  & = \{ \chi_{A*B}(x) \mid A\in \cA,\,B\in\cB \} \\      
  & = \{ \chi_A(x)\oast \chi_B(x) \mid A\in \cA,\,B\in\cB \} \\
  & = \{ \chi_A(x) \mid A\in \cA\} \oast \{\chi_B(x)\mid B\in\cB \} \\      
  & = E_\cA(x) \oast E_\cB(x).
\end{align*}
Employing the notation that was used to define Kleene's extension,
we may write this equation as
\[
  V_{\xi_{\cA\oast\cB}(x)} = V_{\xi_\cA(x)} \oast V_{\xi_\cB(x)},
\]
i.e., $\xi_{\cA\oast\cB}(x) = \xi_\cA(x) \oast \xi_\cB(x)$.
\end{proof}

\begin{proposition} \label{basic:tensor-intpres}
Let $\oast$ be a tensor operator.  Then if $\cA,\cB\osaj\pot(X)$
are intervals, then so is $\cA\oast\cB$, too.
Indeed, if we write $\xi=\xi_{\cA\oast\cB}$, $C_0=\xi^{-1}[\{1\}]$
and $C_1=\xi^{-1}[\{\ud,1\}]$, then $\cA\oast\cB=[C_0,C_1]$.  
\end{proposition}

\begin{proof}
By case (a) of Lemma~\ref{basic:compat} we have that $\cA\oast\cB\osaj [C_0,C_1]$.
Let $C\in[C_0,C_1]$.
As $C\in[C_0,C_1]$, the characteristic function $\chi_C$
is compatible with $\xi_{\cA\oast\cB}$.
By Lemma~\ref{basic:comp-op}, we know that
$\xi_{\cA\oast\cB}=\xi_\cA\oast\xi_\cB$.
This enables us to choose (picking the values $\chi^{(0)}(x)$ and $\chi^{(1)}(x)$
separately for each $x\in X$) functions $\chi^{(0)},\chi^{(1)}\colon X\to\{0,1\}$
such that $\chi_C=\chi^{(0)}\oast\chi^{(1)}$, $\chi^{(0)}$ is compatible with $\xi_\cA$
and $\chi^{(1)}$ is compatible with $\xi_\cB$.  Finally, by case (b) of
Lemma~\ref{basic:compat}, we see that there are $A\in\cA$ and $B\in\cB$
with $\chi_A=\chi^{(0)}$ and $\chi_B=\chi^{(1)}$, which implies
$C=A*B\in\cA\oast\cB$.
\end{proof}

\subsection{Families of teams } \label{families-of-teams}

The general concept of a family of sets arises naturally in numerous
contexts. In this paper our focus is on families of sets arising in
logic, with applications in logic in mind. These families are families
of sets on the base set of the form of a cartesian product
$M^m$. This particular form of the base set permits dimension
computations which do not arise in the abstract setting. In
particular, we can fix $m$ and ask how does the dimension of a family
depend on $|M|$. To avoid trivialities we assume $|M|\ge 2$.

In classical logic one associates with a given formula
$\phi(x_0,\ldots,x_{m-1})$ with the free variables $x_0,\ldots,x_{m-1}$
and a given structure $M$ the set of $m$-tuples
satisfying the formula $\phi$ in $M$:
\[
  \sat{\phi}{M}=\{(a_0,\ldots,a_{m-1})\in M^m\mid M\models
      \phi(a_0,\ldots,a_{m-1})\}.
\]
Such sets of $m$-tuples are called \emph{definable} subsets of
$M^m$. The definable subsets of $M^m$ form a Boolean algebra with
Boolean operations corresponding to the logical operations of first
order logic. The study of this algebra is a well-known method in
logic.

In the same way as classical logic gives rise to definable sets of
$m$-tuples, team semantics and dependence logic (\cite{Vaa}) give rise
to definable families of sets of $m$-tuples. If $M$ is a model, a
\emph{team} in $M$ is a set $T$ of assignments $s$ (i.e. functions) which
map a set $\dom(s)=\{x_0,\ldots,x_{m-1}\}$ of variables, called the
\emph{domain} of $s$ (and of $T$), to $M$. We identify $s$ with the tuple
$(s(x_0),\ldots,s(x_{m-1}))$ and a team with a subset of $M^m$. Every
formula $\phi$ of dependence logic, or another logic based on team
semantics, with free variables in $\vec x=(x_0,\ldots,x_{m-1})$, gives
rise to the set of teams
\begin{equation}\label{satM}
  \sat{\phi}{{M,\vec x}}=\{T\subseteq M^m \mid M\models_T\phi\},
\end{equation}
where $M\models_T\phi$ is the satisfaction relation defined below. We
consider the families $\sat{\phi}{{M,\vec x}}$ a special interesting case of
families of subsets of $M^m$.

If $\ell< m$, there is a canonical projection $M^m\to M^\ell$. We may
identify $T\subseteq M^\ell$ with
$T^*=\{s\in M^m : s\restriction \ell\in T\}$. In this way it is
possible to think of a subset of $M^\ell$ at the same time, via $T^*$,
as a subset of $M^m$, although literally, of course, $T\ne
T^*$.

 Many of
the results of this paper hold for arbitrary families of sets but when
applied to families of the form $\sat{\phi}{{M,\vx}}$, results pertaining to
dependence and independence logics obtain.

In order to make (\ref{satM}) more exact we now recall the inductive
definition of $M\models_T\phi$ from \cite{Vaa}. If $a\in M$, then
$s(a/x)$ is the unique assignment $s'$ such that $s'(x)=a$ and
$s'(y)=s(y)$ for variables $y$ in the domain of $s$ other than $x$. If
$F:T\to \mathcal{P}(M)\jer\{\emptyset\}$, then
$T[F/x]=\{s(a/x)\mid s\in T, a\in F(s)\}$. Finally,
$T[M/x]=\{s(a/x) \mid a\in M, s\in T\}$.

\begin{definition}\label{fol}
\begin{enumerate}[(a)]
  \item $M\models_T\phi$, where $\phi$ is atomic or negated atomic if
    and only if every assignment $s$ in $T$ satisfies $\phi$.
  \item $M\models_T\phi\wedge\psi$ if and only if $M\models_T\phi$ and
    $M\models_T\psi$.
\item $M\models_T\phi\vee\psi$ if and only if there are $U$ and $V$ such that $T=U\cup V$,
    $M\models_U\phi$ and $M\models_V\psi$. (Tensor disjunction)
 \item $M\models_T\phi\,\tland\,\psi$ if and only if there are $U$ and $V$ such that $T=U\cap V$,
    $M\models_U\phi$ and $M\models_V\psi$. (Tensor conjunction)
  \item $M\models_T\phi\ \ivee\ \psi$ if and only if $M\models_T\phi$ or
    $M\models_T\psi$. (Intuitionistic disjunction)
  \item $M\models_T\exists x\phi$ if and only if there is
  $F:T\to \pot(M)\jer\{\emptyset\}$ such that
  $M\models_{T[F/x]}\phi$.
  \item $M\models_T\forall x\phi$ if and only if
    $M\models_{T[M/x]}\phi$.
\end{enumerate}
\end{definition}

This defines $M\models_T\phi$ for every first-order formula $\phi$.
Note that \cite{Vaa} uses only the first two of the four binary connectives in Definition~\ref{fol}. We have kept here the usual notation $\land$ and $\lor$ for these connectives. Intuitionistic disjunction was mentioned in \cite{Vaa} and elaborated on  in \cite{AV}. Tensor conjunction does not seem to have been studied before, and its role is minor here, too. 

By Definition
\ref{fol}(a), for every first-order literal (i.e., atomic or negated
atomic) $\phi$ we have $\sat{\phi}{{M,\vec x}}=[\emptyset,T_\phi]$,
where $T_\phi=\{\vec a\in M^m\mid M\models\phi(\vec a)\}$.  The same
is true if $\phi$ is any formula of first order logic. Thus for first
order $\phi$ the family $\sat{\phi}{{M,\vx}}$ is dominated (by
$T_\phi$), downward closed, convex and supported (by $\emptyset$).

Note further that
for composite $\phi$ the family $\sat{\phi}{{M,\vec x}}$ can be obtained from
the corresponding families for the components $\psi$ of $\phi$ by applying one of the  operators introduced in Example \ref{operator-ex}. For conjunction and (tensor) disjunction we have

\def\arraystretch{1.2}
$$\begin{array}{lcl}
\sat{\phi\wedge\psi}{{M,\vec x}}&=     &\sat{\phi}{{M,\vec x}}\cap\sat{\psi}{{M,\vec x}}  \\
\sat{\phi\vee\psi}{{M,\vec x}}&=     &\sat{\phi}{{M,\vec x}}\lor\sat{\psi}{{M,\vec x}} \\
\end{array}
  $$
Furtheoremore, for tensor conjunction and intuitionistic disjunction we have

\def\arraystretch{1.2}
$$\begin{array}{lcl}
\sat{\phi\,\tland\,\psi}{{M,\vec x}}&=     &\sat{\phi}{{M,\vec x}}\land\sat{\psi}{{M,\vec x}}  \\
\sat{\phi\ \ivee\ \psi}{{M,\vec x}}&=     &\sat{\phi}{{M,\vec x}}\cup\sat{\psi}{{M,\vec x}} \\
\end{array}
  $$

Note however, that in the case of existential and universal quantifiers, the quantified variable needs to be dropped from the tuple $\vec x=(x_0,\ldots,x_{m-1})$:
$$\begin{array}{lcl}
\sat{\exists x_i\phi}{{M,\vx^-}}&=     &\Delta_{\exists i}^{M^m}(\sat{\phi}{{M,\vx}})  \\
\sat{\forall x_i\phi}{{M,\vx^-}}&=     &\Delta_{\forall i}^{M^m}(\sat{\phi}{{M,\vx}}),  \\
\end{array}
  $$
  \def\arraystretch{1}
where $\vx^-$ is the tuple obtained from $\vx$ by deleting the component $x_i$.

We now recall the extension of $M\models_T\phi$ from first order $\phi$ to new non-first order atoms.
Below, the restriction  of a team $T$ to $\vx$, in symbols $T\restriction\vx$, is the set $\{s\restriction\vx:s\in T\}$. We use $\len(\vx)$ to denote the length of the variable (or other) sequence $\vx$.
\begin{definition}\label{atoms}
\begin{enumerate}[(a)]
    \item {\bf Dependence atom}: $M\models_T \dep(\vx,y)$ if and only if $s(\vx)=s'(\vx)$ implies $s(y)=s'(y)$ for all $s,s'\in T$. We allow $\len(\vx)=0$ and call $\dep(y)$ the {\bf constancy atom}. More generally, $M\models_T \dep(\vy)$ if and only if $s(\vy)=s'(\vy)$ for all $s,s'\in T$.
    \item {\bf Exclusion atom}: $M\models_T\vx\ |\ \vy$ if and only if for every $s,s'\in T$  we have $s(\vx)\ne s'(\vy)$. We assume $\len(\vx)=\len(\vy)>0$.

    \item {\bf Inclusion atom}: $M\models_T\vx\subseteq \vy$ if and only if for every $s\in T$ there is $s'\in T$ such that $s(\vx)=s'(\vy)$. We assume $\len(\vx)=\len(\vy)>0$.
    \item {\bf Anonymity atom}: $M\models_T\vx\anonym y$ if and only if for every $s\in T$  there is $s'\in T$ such that $s(\vx)=s'(\vx)$ and $s(y)\ne s'(y)$. We assume that $\vx$ is  non-empty.
    \item {\bf Independence atom}: $M\models_T \vx\perp_{\vz}\vy$ if and only if for every $s,s'\in T$ such that $s(\vz)=s'(\vz)$ there is $s''\in T$ such that $s''(\vz)=s(\vz)$, $s''(\vx)=s(\vx)$ and $s''(\vy)=s'(\vy)$. We assume that $\vx$ and $\vy$ are non-empty.
The atom
$\vx\perp\vy$, corresponding to the case $\vz$ is empty, is called the \emph{pure} independence atom, while  $ \vx\perp_{\vz}\vy$ is otherwise called the \emph{conditional} independence atom.
    
    \item {\bf The general concept of an atom:} Suppose $C$ is a class, closed under isomorphisms,  of pairs $(A,T)$ where $A$ is a set and $T$ is a team in $A$ with domain $\vec{x}$. We can associate with $C$ a new atom $\alpha_C(\vec{x})$ and define  $M\models_T\alpha_C(\vec{x})$ to hold if and only if $(A,T\restriction\vec{x})\in C$, where $A$ is the domain of the model $M$.
\end{enumerate}
\end{definition}
\medskip

By closing the respective atom under the logical operations (b), (c), (f) and (g) of Definition~\ref{fol} we obtain \emph{dependence logic}, \emph{constancy logic}, \emph{exclusion logic}, \emph{inclusion logic}, \emph{anonymity logic} and \emph{(pure or conditional) independence logic}.

Note that we defined $\dep(\vx,y)$ for single variable $y$ only. This is because $\dep(\vx,\vy)$ for a vector $\vy=(y_1\ldots,y_n)$, which we adopt now as a shorthand, can be defined as 
$$\dep(\vx,y_1)\wedge\ldots\wedge\dep(\vx,y_n).$$ We use the same convention for $\dep(\vx)$.

If $\phi$ is a dependence atom or an exclusion atom, then $\sat{\phi}{{M,\vx}}$ is downward closed and supported by $\emptyset$ but not necessarily closed under unions or dominated. If $\phi$ is an inclusion atom or an anonymity atom, then $\sat{\phi}{{M,\vx}}$ is closed under unions and dominated by $M^{\len(\vx)}$ but not necessarily downward closed or supported. 

\begin{example}\label{halfandparity}
An example of a sentence combining  dependence atoms and logical operations is the following formula which is satisfied by a team $T$ in a model of size $n$ if and only if
$|\hspace{1pt} T\restriction \vx \hspace{1pt}|\le n^k/2$, where $\len(\vx)=k$:
$$
\exists\vy(\dep(\vy,\vx)\wedge \vx\ |\ \vy).$$ Here $\len(\vy)=k$.
An example of a sentence combining  a number of different  atoms as well as logical operations is the following formula which is satisfied by a team $T$  if and only if
$|\hspace{1pt} T\restriction \vx \hspace{1pt}|$ is even:$$\begin{array}{l}
\exists u\exists v\exists\vy\exists\vz(\vy\vz\perp \vx\wedge\vy\subseteq\vx\wedge\vz\subseteq\vx\\
\qquad\wedge((u=v\wedge\vx\subseteq\vy)\vee(u\ne v\wedge \vx\subseteq\vz))\\
\qquad\wedge\ \vy\ |\ \vz\ \wedge \dep(\vz,\vy)\ \wedge \dep(\vy,\vz))
\end{array}.$$

\end{example}

We will also consider the extension of first-order logic with Lindström quantifiers (see \cite{Li} for definition).
For the sake of simplicity, we restrict attention to Lindström quantifiers of type $(r)$ for some positive integer
$r$ (i.e., quantifiers binding a single formula). Such a quantifier $Q_\cK$ is associated to any 
isomorphism closed class $\cK$ of structures of the form $(A,R)$, $R\subseteq A^r$. If $\psi$ is a formula and
$\vec y$ is an $r$-tuple of variables, then applying the quantifier $Q_\cK$ we obtain a new formula
$Q_\cK\vec y\,\psi$ in which all occurrences of the variables in $\vec y$ are bound.

To define the team semantics of $Q_\cK$, we adapt the notation used for existential quantifier:
if $F\colon T\to \pot(M^r)$, then $T[F/\vec y]=\{s(\vec b/\vec y)\mid s\in T, \vec b\in F(s)\}$.  

\begin{definition}\label{Linquantsem}
\begin{enumerate}
    \item[] $M\models_T Q_\cK\vec y\,\psi$ if and only if there exists $F\colon T\to \pot(M^r)$ such that 
    $M\models_{T[F/\vec y]}\psi$ and $(M,F(s))\in \cK$ for all $s\in T$.
\end{enumerate}
\end{definition}

The semantics of Lindström quantifiers can also be formulated in terms of operators that map sets of the form $\sat{\psi}{{M,\vz}}$ to sets $\sat{Q_\cK\vy\,\psi}{{M,\vx}}$, where $\vz$ consists of the variables in the tuples $\vx$ and $\vy$. To work out the details of these operators, we fix a quantifier $Q_\cK$ of type $(r)$, the length $m\ge r$ of $\vz$, the tuple $\vec\ell=(l_0,\ldots,\ell_{r-1})$ for which $\vy=(z_{\ell_0},\ldots,z_{\ell_{r-1}})$, and the universe $M$ of the model. Note that there is no reason to assume that the components of $\vec\ell$ are in ascending order; the quantifier $Q_\cK$ can be applied to any $r$-tuple of distinct variables in $\vz$. On the other hand, we can assume w.l.o.g. that $\vx$ lists the rest of the variables in $\vz$ in ascending order, i.e., for each $i<m-r$, $x_i=z_j$, where $j\not\in\{\ell_0,\ldots,\ell_{r-1}\}$ and $i=|\{k<j\mid k\notin\{\ell_0,\ldots,\ell_{r-1}\}\}|$. Thus, $\vz$ is obtained from $\vx$ and $\vy$ by re-ordering the latter and shuffling according to $\vec\ell$. We use the notation $\vz=\vx\otimes_{\vec\ell}\vy$ to denote this shuffling, and similarly $\vc=\va\otimes_{\vec\ell}\vb$ for tuples $\vc,\va,\vb$ of elements in $M$.

Assume then that $S\subseteq M^m$ is a team with domain $\{z_0,\ldots,z_{m-1}\}$. For each $\va\in M^{m-r}$ the \emph{$\vec\ell$-restriction of $S$ on $\va$} is the set $S[\va]_{\vec\ell}:=\{\vb\in M^r\mid \va\otimes_{\vec\ell}\vb\in S\}$. Furtheoremore, the \emph{$(\cK,\vec\ell)$-projection of $S$} is the set $\pi_{\cK,\vec\ell}\,(S):=\{\va\in M^{m-r}\mid (M,S[\va]_{\vec\ell})\in\cK\}$. The idea here is that if $T=\pi_{\cK,\vec\ell}\,(S)$ for some team $S\in\sat{\psi}{{M,\vz}}$, then defining $F\colon T\to\pot(M^r)$ by $F(\va)=S[\va]_{\vec\ell}$ for each
$\va\in T$, we see that the truth condition of Definition \ref{Linquantsem} holds for the team $T$, provided that $S=T[F/\vy]$. It is clear that $T[F/\vy]\subseteq S$, and the converse inclusion holds if and only if $\{\va\in M^{m-r}\mid S[\va]_{\vec\ell}\not=\emptyset\}\subseteq T$. We say that $T$ is the \emph{proper $(\cK,\vec\ell)$-projection of $S$}, in symbols $T=\pi^p_{\cK,\vec\ell}\,(S)$, if this condition holds.

The argument above shows that if $T=\pi^p_{\cK,\vec\ell}\,(S)$ for some team $S\in\sat{\psi}{{M,\vz}}$, then $T\in\sat{Q_\cK\vy\,\psi}{{M,\vx}}$. Assuming that $(M,\emptyset)\notin\cK$, the converse implication is also true. Indeed, if $M\models_T Q_\cK\vec y\,\psi$, then there is a function $F\colon T\to \pot(M^r)$ such that $M\models_{T[F/\vy]}\psi$ and $(M,F(\va))\in \cK$ for all $\va\in T$. Since $(M,\emptyset)\notin\cK$, $F(\va)\not=\emptyset$ for every $\va\in T$, whence $T=\pi_{\cK,\vec\ell}\,(T[F/\vy])$. Moreover, the condition $\{\va\in M^{m-r}\mid S[\va]_{\vec\ell}\not=\emptyset\}\subseteq T$ clearly holds for any team $S$ of the form $T[F/\vy]$. Thus, we see that $T=\pi^p_{\cK,\vec\ell}\,(T[F/\vy])$.

Note however, that if $(M,\emptyset)\in\cK$, the argument for $T=\pi_{\cK,\vec\ell}\,(T[F/\vy])$ fails: by the definition we always have $\pi_{\cK,\vec\ell}\,(T[F/\vy])\subseteq T$, but if $F(\va)=\emptyset$ for some $\va\in T$, then $\va\notin \pi_{\cK,\vec\ell}\,(T[F/\vy])$. In this case the correct condition for a team $T$ being in the family $\sat{Q_\cK\vy\,\psi}{{M,\vx}}$ is that there exist  teams $S\in\sat{\psi}{{M,\vz}}$ and $T'\subseteq T$ such that $T'=\pi^p_{\cK,\vec\ell}\,(S)$.

We are now ready to define the operators on families of teams corresponding to Lindström quantifiers.

\begin{definition}\label{Q-oper}
The \emph{$(\cK,\vec\ell)$-projection operator}  $\Delta^{M^m}_{\cK,\vec\ell}\colon\pot(\pot(M^m))\to\pot(\pot(M^{m-r}))$ is defined separately in two cases.
\begin{itemize}
    \item 
If $(M,\emptyset)\notin\cK$, then for each $\cA\in\pot(\pot(M^m))$,
$$
    \Delta^{M^m}_{\cK,\vec\ell}(\cA)=\{B\in\pot(\pot(M^{m-r}))\mid B=\pi^p_{\cK,\vec\ell}\,(A)\text{ for some }A\in\cA\}.
$$
    \item 
If $(M,\emptyset)\in\cK$, then for each $\cA\in\pot(\pot(M^m))$,
$$
    \Delta^{M^m}_{\cK,\vec\ell}(\cA)=\{B\in\pot(\pot(M^{m-r}))\mid \pi^p_{\cK,\vec\ell}\,(A)\subseteq B\text{ for some }A\in\cA\}.
$$
\end{itemize}
\end{definition}

By the argument given before Definition \ref{Q-oper}, the operator $\Delta^{M^m}_{\cK,\vec\ell}$ captures the semantics of the quantifier $Q_\cK$:
$$
\begin{array}{lcl}
	\sat{Q_\cK \vec y\,\psi}{{M,\vx}}&=     &\Delta_{\cK,\vec\ell}^{M^m}(\sat{\psi}{{M,\vz}}).
\end{array}
$$

Note further that the standard existential and universal quantifiers are special cases of Lindström quantifiers: $\exists=Q_{\cK_\exists}$ for the class $\cK_\exists=\{(A,R)\mid R\subseteq A, R\not=\emptyset\}$, and $\forall=Q_{\cK_\forall}$ for $\cK_\forall=\{(A,R)\mid R= A\}$. Thus, the corresponding operators are also identical: for each $i<m$, $\Delta_{\exists i}^{M^m}= \Delta_{{\cK_\exists},\vec\ell}^{M^m}$ and $\Delta_{\forall i}^{M^m}= \Delta_{{\cK_\forall},\vec\ell}^{M^m}\,$, where $\vec\ell=i$ (i.e., $\ell$ is of length $1$, and $\ell_0=i$). For this reason there is no need to consider the operators $\Delta_{\exists i}^{M^m}$ and $\Delta_{\forall i}^{M^m}$ separately in the sequel.
\medskip

\begin{remark}
Note that if $(M,\emptyset)\in\cK$ and $M\models_\emptyset\psi$,
then $M\models_T Q_\cK\vec y\,\psi$ for any team $T$. Indeed, if
$F\colon T\to \pot(M^r)$ is the function with $F(s)=\emptyset$ for
all $s\in T$, then $T[F/\vec y]=\emptyset$, whence the truth
condition in Definition \ref{Linquantsem} holds. Every formula
$\psi$ in the extension of first-order logic by the atoms listed in
Definition \ref{atoms} has the Empty Team Property:
$M\models_\emptyset\psi$ holds for all models $M$ (see
\cite{Vaa}). It is easy to see that the same holds also if we add
arbitrary Lindström quantifiers to the logic. In fact, with the
exception of logics with the non-empty atom (see Definition
\ref{non-empty-atom}), all the logics we consider in this paper have the empty
team property. Thus we see that a quantifier $Q_\cK$ becomes trivial
(on $M$) in our setting if $(M,\emptyset)\in \cK$, as in this case
$M\models_T Q_\cK\vec y\,\psi$ holds for every team $T$ and every formula
$\psi$.
\end{remark}


%

\section{Dimension calculations}

In this section we compute exact values, or in some cases just upper
and lower bounds, to upper, dual and cylindrical dimensions of some
important concrete examples of families of sets. This will be used
later to estimate dimensions of definable families of teams in various
logics built around the atoms of Definition~\ref{atoms}.

\subsection{Convex shadows and hulls}

In this subsection, we develop some auxiliary tools useful in  concrete
dimension calculations.  In particular, we introduce the notions
of convex shadow and the dual notion of dual convex shadow,
which facilitate the calculation of upper and dual upper
dimension of a given family.
The point is, that when we need to check if a subfamily dominates
the given family,  the convex shadows are the canonical
dominated convex families we need to relate to the sets in the
dominating family.

\begin{definition}
Let $\cA$ be a family of sets and $A\in\cA$.
The \emph{convex shadow of $A$  in the family $\cA$} is the family
\[
  \varjo_A(\cA)=\{ B\osaj A \mid [B,A]\osaj\cA \}.
\]
Similarly, the \emph{dual convex shadow} of $A$ in $\cA$ is
\[
  \varjo^A(\cA)=\{ B\in\cA \mid A\osaj B,\,[A,B]\osaj\cA \}.
\]
A set $A\in\cA$ is called \emph{critical} in $\cA$
if its convex shadow is maximal in the family
\[
  \{ \varjo_B(\cA) \mid B\in\cA \}.
\]
We use the notation
\[
  \Crit(\cA) = \{ A\in\cA \mid A\text{ critical in $\cA$} \}.
\]
Similarly, we define the notion of \emph{dual criticality}.
We denote the family of dually critical sets in $\cA$ by $\Critd(\cA)$:lla.
\end{definition}

\begin{lemma} \label{dimcalc:shadowcrit}
Let  $\cA$ be a family of sets and $A\in\cA$.
\begin{enumerate}[(a)]
\item
$\varjo_A(\cA)$ is the largest dominated convex family $\cC\osaj\cA$
with $\bigcup\cC=A$. Similarly, $\varjo^A(\cA)$
is the largest supported convex family $\cC\osaj\cA$ with $\bigcap\cC=A$.

\item
A family $\cG\osaj\cA$ dominates $\cA$ if and only if $\bigcup_{G\in\cG}\varjo_G(\cA)=\cA$.  
Dually, $\cH\osaj\cA$ supports $\cA$ if and only if $\bigcup_{H\in\cH}\varjo^H(\cA)=\cA$.  

\item
If $\cG$ dominates $\cA$, then $\Max(\cA)\osaj\cG$, and
 if  $\cH$ supports $\cA$, then $\Min(\cA)\osaj\cH$.

\item
Suppose that the family of families $\{\varjo_A(\cA) \mid A\in\cA\}$
satisfies Zorn condition.
Then there is a family $\cG$ dominating $\cA$
such that  $\cG\osaj\Crit(\cA)$ and $|\cG|= \DD(\cA)$.
The dual result also holds.

\end{enumerate}
\end{lemma}

\begin{proof}
The proofs of the dual claims are similar to the primary claims,
so we shall skip them.
  
(a) Clearly, $\varjo_A(\cA)$ is a dominated convex subfamily of $\cA$
with $\bigcup\cA=A$.  Suppose $\cC\osaj\cA$ is another dominated
convex subfamily with $\bigcup\cC=A$.  For $C\in\cC$, we have
$[C,A]\osaj\cC\osaj\cA$ by convexity of $\cC$, so $\cC\osaj\varjo_A(\cA)$.

(b) If $\bigcup_{G\in\cG}\varjo_G(\cA)=\cA$, then the dominated convex
families $\varjo_G(\cA)=\cA$, $G\in\cG$, witness that $\cG$
dominates $\cA$.

Conversely, suppose that $\cG$ dominates $\cA$ and
the families $\cD_G$, $G\in\cG$,
witness that (meaning that $\bigcup_{G\in\cG}\cD_G=\cA$ and $\bigcup\cD_G=G$,
for each $G\in\cG$).  Then by the preceding claim, we have that
for every $G\in\cG$, $\cD_G\osaj\varjo_G(\cA)$, implying
$\bigcup_{G\in\cG}\varjo_G(\cA)$.

(c)  Suppose $\cG$ dominates $\cA$ and $M$ is maximal in the family $\cA$.
Then by case~a, we have $M\in\varjo_G(\cA)$, for some $G\in\cG$.
However, $M\in\varjo_G(\cA)$ implies $M\osaj G$, so by maximality of $M$,
we have $M=G\in\cG$. Consequently, $\Max(\cA)\osaj\cG$.

(d) Pick a subfamily $\cG_0\osaj\cA$ dominating $\cA$ such that
$|\cG_0|=\DD(\cA)$.  Let $\mathbf{S}=\{\varjo_A(\cA) \mid A\in\cA\}$.
By assumption, $\mathbf{S}$ satisfies the Zorn condition, so for
each $G\in\cG_0$, there is a $D_G\in\cA$ with maximal convex shadow
such that $\varjo_G(\cA)\osaj \varjo_G(\cA)$.  In other words,
there is a critical $D_G\in\cA$ with $\varjo_G(\cA)\osaj \varjo_G(\cA)$.
Put $\cG=\{D_G \mid G\in\cG \}$.
As $\cG_0$ dominates $\cA$, we have $\bigcup_{G\in\cG_0}\varjo_G(\cA)=\cA$,
which clearly impies 
$$
\bigcup_{D\in\cG}\varjo_D(\cA)=\bigcup_{G\in\cG_0}\varjo_{D_G}(\cA)=\cA.
$$
Hence, $\cG$ also dominates $\cA$, and $|\cG|\le|\cG_0|=\DD(\cA)$,
so $|\cG|=\DD(G)$.
\end{proof}

Convex shadows and dual convex shadows are maximal \textsl{subfamilies}
satisfying the appropriate properties.  The similar operations that produce
superfamilies instead of subfamilies are called hulls.  We shall utilize
these latter concepts in later sections.

\begin{definition}
Let $\cA$ be a nonempty family of sets.
The \emph{convex hull} of $\cA$ is
\[
  \CH(\cA)=\bigcup_{A,A'\in\cA}[A,A'], 
\]
the \emph{dominated (convex) hull} of $\cA$ is
\[
  \DCH(\cA)=\bigcup_{A\in\cA}\left[A,\bigcup\cA\right], 
\]
and the \emph{supported (convex) hull} is 
\[
  \SCH(\cA)=\bigcup_{A\in\cA}\left[\bigcap\cA,A\right].
\]
We set also $\CH(\tyj)=\tyj$.
\end{definition}

Note that there is no unique least dominated, or supported,
convex family containing the empty family (all the singletons do).

If the family of sets is finite, we may drop the braces from the notation
in the customary manner, writing $\DCH(A_0,\ldots,A_{k-1})$
instead of $\DCH(\{A_0,\ldots,A_{k-1}\})$, or 
$\SCH(A_0,\ldots,A_{k-1})$ instead of $\DCH(\{A_0,\ldots,A_{k-1}\})$.
Note the special cases $\DCH(A,B)=[A,A\cup B]\cup[B,A\cup B]$
and $\SCH(A,B)=[A\cap B,A]\cup[A\cap B,B]$.

We omit the proof of the following lemma, as it is straightforward.

\begin{lemma}
Let $\cA$ be a family of sets.  Then
\begin{enumerate}[(a)]
\item   $\CH(\cA)$ is the least convex family containing $\cA$,
\item   $\DCH(\cA)$ is the least dominated convex family containing $\cA$ and
\item   $\SCH(\cA)$ is the least supported convex family containing $\cA$.
\end{enumerate}
\end{lemma}

\subsection{Dimensions of particular families}

In this subsection, we calculate the dimensions of some concrete
families of sets that are relevant to team semantics but certainly are
familar from other contexts, too.

For non-empty finite base sets $X$ and $Y$, here is a list of families that
we consider:
\begin{align*}
\cF &= \{ f\osaj X\times Y \mid f\text{ is a mapping }\}, \\
\cX &=\{ R\osaj X\times X \mid \dom(R)\cap\rg(R)=\tyj \}\\
\cI_{\osaj} &=\{ R\osaj X\times X \mid \dom(R)\osaj \rg(R) \}, \\
\cY &=\{ R\osaj X\times Y \mid \text{$R$ is anonymous} \},\\
\cI_{\perp} &=\{ A\times B \mid A\osaj X,\,B\osaj Y \},
\end{align*}
where we call a relation $R\osaj X\times Y$ \emph{anonymous} if
for all $x\in\dom(R)$ there exist distinct $y,y'\in Y$
with $(x,y),(x,y')\in R$.

We calculate the dimensions of these families with the aid of shadows
and critical sets.  The families $\cF$ and $\cX$ are the easiest
cases, as they are downward closed.  We handle each of the other
families in a separate lemma of its own.

\begin{lemma} \label{dimcalc:indep}
Suppose $|X|,|Y|\ge 2$. Let $A\osaj X$ and $B\osaj Y$. Then:
\begin{enumerate}[(a)]
\item
If  $|A|,|B|\ge 2$,  we have
$\varjo_{A\times B}(\cI_{\perp})=\varjo^{A\times B}(\cI_{\perp})=\{A\times B\}$.

\item
If $|A|\le 1$ or $|B|\le 1$, then
$\varjo_{A\times B}(\cI_{\perp})=\pot(A\times B)$.

\item
If $|A|=1$ and $|B|\ge 2$, then
\[
   \varjo^{A\times B}(\cI_{\perp})=\{ A\times D \mid B\osaj D\osaj Y \}.
\]
Similarly, $|A|\ge 2$ and $|B|=1$ implies
$\varjo^{A\times B}(\cI_{\perp})=\{ C\times B \mid A\osaj C\osaj X \}$.

\item
If $|A|=|B|=1$, then $\varjo^{A\times B}(\cI_{\perp})$ consists of
set of the form $A\times D$ and $C\times B$ with $A\osaj C\osaj X$
and $B\osaj D\osaj Y$.

\item
$\varjo^\tyj(\cI_{\perp})$ consists of all the sets $C\times D$
where $C\osaj X$, $D\osaj Y$ and $|C|\le 1$ or $|D|\le 1$.

\item
The critical and dual critical families of $\cI_{\perp}$ are
\begin{align*}
  \Crit(\cI_{\perp}) =&\{A\times B\osaj X\times Y \mid |A|\ge 2,|B|\ge 2\} \\   
                     &\cup \{ \{a\}\times Y \mid a\in X \}\\
                     &\cup \{ X\times \{b\} \mid b\in Y \}. \\
  \Critd(\cI_{\perp})=&\{A\times B\osaj X\times Y \mid |A|\ge 2,|B|\ge 2\}
                       \cup \{\tyj\}. \\   
\end{align*}

\item
For each $R\in\cI_{\perp}$, the shadow $\varjo_R(\cI_\perp)$ is
an interval.
\end{enumerate}

\end{lemma}

\begin{proof}
(a) If $|A|\ge 2$ and $|B|\ge 2$, any addition or deletion of
a point $(x,y)$ to or from $A\times B$ results to a set that
is not a cartesian product of the form $A'\times B'$, which
implies the result.

(b) The claim is trivial if either of the sets $A$ or $B$ is empty,
so assume by symmetry that $A=\{a\}$.  Then every subset of $A\times B$
can be written as $A\times B'$ for some $B'\osaj B$, so
$\varjo_{A\times B}(\cI_{\perp})=\pot(A\times B)$.

(c) Suppose $A=\{a\}$ and $|B|\ge 2$.  Let $R\in\cI_{\perp}$ with
$R\supseteq A\times B$. Write $R=C\times D$  where
$C\osaj X$ and $D\osaj Y$ with $C\supseteq A$ and $D\supseteq B$.
If $C\neq A$, it is easy to see that
$C\times D\not\in\varjo^{A\times B}(\cI_{\perp})$ as we can
pick $c\in C\jer A$ and $b\in B$, whence
$A\times B\osaj (A\times B)\cup\{c,d\}\osaj C\times D$,
but $(A\times B)\cup\{c,d\}\not\in\cI_{\perp}$.
In contrast, for every $D\osaj Y$ with $D\supseteq B$
we have $[A\times B,A\times D]\osaj\cI_{\perp}$ as $A$ is
a singleton.

For items (d) and (e), the proof is quite similar to the proof of item (c).

(f)  Consider first the set $A\times B$ where $|A|\ge 2$ and
$|B|\ge 2$.  Then by items (a)--(e), the only convex shadow or
dual convex shadow that covers $A\times B$ is the shadow or 
dual shadow of $A\times B$ itself.  Hence, $A\times B$ is both
critical and dual critical.

Consider then the case $A=\{a\}$ ($a\in X$) is a singleton.
By item (b), among the sets $A\times B$ the set $A\times Y$
has the largest shadow, including the case $B=\tyj$.
The symmetric case when $B$ is a singleton and $A$ varies
is handled in the same way.  For the dual case, items (c)-(e)
show that the empty set has the largest dual shadow.

(g) This follows from items (a) and (b),  as we observe
that $\{A\times B\}=[A\times B,A\times B]$
and $\pot(A\times B)=[\tyj,A\times B]$.
\end{proof}

\begin{lemma}  \label{dimcalc:incl}
Assume that $m=|X|\ge 2$. Let $R,S\in \cI_{\osaj}$.

\begin{enumerate}[(a)]
\item
If $S\osaj R$, then  
\[
   [S,R] \osaj \cI_{\osaj} \text{ if and only if } \dom(R\jer\id_X)\osaj\rg(S).
\]     

\item
We have
\[
  \varjo_R(\cI_{\osaj})=\{ T\in\cI_{\osaj} \mid T\osaj R,\,A\osaj\rg(T)\}
\]
where $A=\dom(R\jer\id_X)$. 

\item
For $A\osaj X$, put
$R_A=(A\times X)\cup\id_X$.
Then we have that $R_A\in\cI_{\osaj}$ and
\[
  \varjo_{R_A}(\cI_{\osaj})
  = \{ T\in\cI_{\osaj} \mid \dom(T\jer\id_X)\osaj A\osaj\rg(T) \}.
\]  

\item
We have $\Crit(\cI_{\osaj})= \{ R_A \mid A\osaj X\}$.

\item
$\Crit(\cI_{\osaj})\jer\{R_{\{a\}} \mid a\in X \}$ is a family
of smallest size that dominates $\cI_{\osaj}$.

\item
Denote $B=\rg(R)$. Then
\[
  \varjo^R(\cI_{\osaj})
  =\{ T\in\cI_{\osaj} \mid R\osaj T,\,\dom(T\jer\id_X)\osaj B \}
  =[R,R_B].
\]

\item  
$\Critd(\cI_{\osaj})
=\{ R\in\cI_{\osaj} \mid R^{-1}\text{ is a mapping}\}$.

\item
$\Critd(\cI_{\osaj})\jer\{{\{(a,a)\}} \mid a\in X \}$ is the
smallest family that supports $\cI_{\osaj}$.
\end{enumerate}

\end{lemma}

\begin{proof}
(a)  Suppose $\dom(R\jer\id_X)\osaj\rg(S)$ and consider $T\in[S,R]$, i.e.,
$S\osaj T\osaj R$.  Let $x\in\dom(T)$.  Pick $y$ such that $(x,y)\in T$.
If $x=y$, then trivially $x\in\rg(T)$.  Otherwise $x\neq y$, so
\[
  x\in\dom(T\jer\id_X)\osaj\dom(R\jer\id_X)\osaj\rg(S)\osaj\rg(T).
\]
Thus in both cases, we have $x\in\rg(T)$, so $\dom(T)\osaj\rg(T)$.
Hence $T\in\cI_{\osaj}$, and $[S,R]\osaj\cI_{\osaj}$.

Suppose to the contrary that $\dom(R\jer\id_X)\not\osaj\rg(S)$.
Then we may choose $x\in\dom(R\jer\id_X)\jer\rg(S)$.
Pick $y\neq x$ with $(x,y)\in R$, and consider $T=S\cup\{(x,y)\}$.
Clearly, $x\in\dom(T)$, but $x\not\in\rg(T)=\rg(S)\cup\{y\}$,
so $T\not\in\cI_{\osaj}$.  This proves that $[S,R]\not\osaj\cI_{\osaj}$.

(b) This is a direct application of the previous item.

(c) We  first note that $\dom(R_A)=X=\rg(R_A)$, as $\id_X\osaj R_A$,
implying that $R_A\in\cI_{\osaj}$.  It is easy to see that
$T\osaj R_A$ if and only if $\dom(T\jer\id_X)\osaj A$, so the
latter result follows from the preceding item.

(d)  To prove that each critical set in $\cI_{\osaj}$
is of the form $R_A$, for some $A\osaj X$, let $R\in\cI_{\osaj}$.
Denote $A=\dom(R\jer\id_X)$.  One easily sees that $R\osaj R_A$,
and now the previous items imply that
$\varjo_R(\cI_{\osaj})\osaj\varjo_{R_A}(\cI_{\osaj})$.
Consequently, it is enough to show that the shadows of the sets $R_A$,
$A\osaj X$, are incomparable.  Let $A,A'\osaj X$, $A\neq A'$.
Suppose first that $|A|\ge 2$. Let $f$ be any permutation of $A$
without fixed points.  Then $\dom(f\jer\id_X)=\dom(f)=A=\rg(A)$,
so the preceding item implies that $f\in\varjo_{R_A}(\cI_{\osaj})$,
but $f\not\in\varjo_{R_{A'}}(\cI_{\osaj})$.  If $A=\tyj$,
we see similarly that
$\tyj\in\varjo_{R_A}(\cI_{\osaj})\jer\varjo_{R_{A'}}(\cI_{\osaj})$.
Now suppose $A=\{a\}$ is a singleton.  Then there is $b\in X$, $b\neq a$,
such that $A'\neq\{a,b\}$. Consider $T=\{(a,b),(a,a)\}$.
Then $\dom(T)=\{a\}\osaj\{a,b\}=\rg(T)$, whence
$T\in\varjo_{R_A}(\cI_{\osaj})$ but $T\not\in\varjo_{R_{A'}}(\cI_{\osaj})$.

(e) By Lemma~\ref{dimcalc:shadowcrit}, we know that
$\Crit(\cI_{\osaj})$ dominates $\cI_{\osaj}$, and we can find a
dominating family of the smallest size from the collection of
its subfamilies.  Now the proof of the preceding item actually
show that $\Crit(\cI_{\osaj})\jer\{R_A\}$ does not dominate $\cI_{\osaj}$,
for any $A\osaj X$ which is not a singleton.  However,
$\Crit(\cI_{\osaj})\jer\{R_{\{a\}} \mid a\in X \}$ does,
as we see from the following:  Let $a\in X$ and let $R\in\cI_{\osaj}$
be any set with $a\in\dom(R\jer\id_X)$. Pick $b\neq a$ with $(a,b)\in R$.
Then $a\in\dom(R)\osaj\rg(R)$, but also $b\in\rg(R)$, so $a,b\in\rg(R)$.
Item c now shows that $R\in R_A$ where $A=\rg(R)$, and $|A|\ge 2$.

(f)  The first equality is a direct consequence of item~a.
For the second equality, it is enough to observe that $R_B$
is, by definition, the largest $T\osaj X\times X$
such that $\dom(T\jer\id_X)\osaj B$.

(g)  Assume that $R\osaj S$ and $\rg(R)=\rg(S)=B$
hold for $R,S\in\cI_{\osaj}$.  Then by the previous item, we have
that $\varjo^S(\cI_{\osaj})=[S,R_B]\osaj[R,R_B]\osaj\varjo^R(\cI_{\osaj})$.
Consequently, $R$ can have a maximal dual shadow only if $R\in\cI_{\osaj}$
is minimal among all the relations having the same range, i.e.,
if $R^{-1}$ is a mapping.

Let us check that this condition is also
sufficient, i.e., if $R\in\cI_{\osaj}$ and $R^{-1}$ is a mapping, then
$R$ has a maximal dual shadow among
the dual shadows $\varjo^S(\cI_{\osaj})$, for $S\in\cI_{\osaj}$.
We need to consider only the case when $S\subsetneq R$ and
$S^{-1}$ is also a mapping.  But then
$\rg(S)=\dom(S^{-1})\subsetneq \dom(R^{-1})=\rg(R)=B$,
which implies that $\dom(R_B\jer\id_X)=B\not\osaj\rg(S)$.
Hence, $R_B\not\in\varjo^S(\cI_{\osaj})$ and
consequently $\varjo^R(\cI_{\osaj})\not\osaj\varjo^S(\cI_{\osaj})$.
This means that $R$ has a maximal dual shadow.
 
(h)  Denote $\cC=\Critd(\cI_{\osaj})\jer\{{\{(a,a)\}} \mid a\in X \}$.
We first show that $\cC$ supports $\cI_{\osaj}$.
Let $R\in\cI_{\osaj}$. If $\rg(R)$ is a singleton, say, $\rg(R)=\{a\}$,
then we must have $R=\{(a,a)\}$ and $R\in\varjo^{\tyj}(\cI_{\osaj})$,
where $\tyj\in\cC$. Otherwise, we select any $R_0\osaj R$ with
$(R_0)^{-1}$ a mapping and $\rg(R_0)=\rg(R)$ and observe that
$R_0\in\cC$.  The rest is proved similarly as above.
%
%
\end{proof}

\begin{lemma}  \label{dimcalc:anonym}
Assume $|Y|\ge 2$.
\begin{enumerate}[(a)]
\item
  Let $R\osaj R'\osaj X\times Y$. Then
\[
  [R,R']\osaj\cY \Leftrightarrow R,R'\in\cY\land \dom(R)=\dom(R').
\]  

\item $\Crit(\cY)=\{ A\times Y \mid A\osaj X\}$ is the smallest family
that dominates $\cY$.

\item  $\Critd(\cY)=\{ f\cup g \mid A\osaj X,f,g\colon A\to Y,\forall x\in
       A: f(x)\neq g(x)\}$ is the smallest family that supports $\cY$.
\end{enumerate}
\end{lemma}

\begin{proof}
(a) This is obvious from the definition of $\cY$.

(b) Consider the following shadowing relation $\sqsubseteq$
between elements of $\cY$:  $R\sqsubseteq S$ if and only if
$R\in\varjo_S(\cY)$.  This appeared actually already in the
previous item, so for $R,S\in\cY$, it holds that
$R\sqsubseteq S$ if and only if $R\osaj S$ and $[R,S]\osaj\cY$
if and only if $R\osaj S$ and $\dom(R)=\dom(S)$.
It is immediate that $\sqsubseteq$ is a partial ordering on $\cY$.
Thus, if $R\sqsubseteq S$, then $\varjo_R(\cY)\osaj\varjo_S(\cY)$.
This implies that, in order to an element of $\cY$ be
critical, it must be maximal with respect to $\sqsubseteq$.
It is easy to see that these maximal elements are of the form $A\times Y$,
for some $A\osaj X$.  As each $R\in\cY$ is also included in
the set $\dom(X)\times Y$ for which $R\sqsubseteq \dom(X)\times Y$,
we also see that $\{A\times Y \mid A\osaj X\}$ dominates $\cY$.
$A\times Y$ is certainly critical, as adding any $(c,d)\not\in A\times Y$
to $A\times Y$ destroys anonymity.

(c)  By \ref{dimcalc:shadowcrit} item~d,
$\Critd(\cY)$ supports the family~$\cY$.
Studying the shadowing relation further, we observe that
for $R,S\in\cY$, we have that $R\sqsubseteq S$ if and only if
$S\in\varjo^R(\cY)$. Consequently, dual critical sets are those which
are minimal with respect to the shadowing relation. These are
exactly the sets of form $f\cup g$ where $f,g\colon A\to Y$ and
for all $x\in A$ we have $f(x)\neq g(x)$.  Clearly all such sets
have to be included in a supporting family (to support themselves),
so $\Critd(\cY)$ is the smallest family that supports $\cY$.
\end{proof}

\begin{theorem}\label{dim-calculations}
Let  $X$ and $Y$ be finite base sets with $\ell=|X|\ge 2$ and $n=|Y|\ge 2$.
Then:
\begin{align*}
&\DD(\cF)=n^\ell, && \DDd(\cF)=1,&& \CD(\cF)=\DD(\cF), \\ 
&\DD(\cX)=2^\ell-2, && \DDd(\cX)=1, && \CD(\cX)=\DD(\cX), \\
&\DD(\cI_{\osaj})=2^\ell-\ell,
      && \DDd(\cI_{\osaj})=1+\sum^\ell_{k=2}\binom{\ell}{k}k^k,
                         && \CD(\cI_{\osaj})=\DDd(\cI_{\osaj}), \\
&\DD(\cY)=2^\ell,   &&\DDd(\cY)=\sum^\ell_{k=0}\binom{\ell}{k}\binom{n}{2}^k ,
  && \CD(\cY)=\DDd(\cY).\\
  &\DD(\cI_{\perp})= \scriptstyle (2^\ell-\ell-1)(2^n-n-1)+\ell+n,
  && \DDd(\cI_{\perp})= \scriptstyle (2^\ell-\ell-1)(2^n-n-1)+1,
  && \CD(\cI_{\perp})=\DD(\cI_{\perp}), \\   
\end{align*}
\end{theorem}

\begin{proof}
Observe first that the family~$\cF$ is downwards closed, so it is
trivially supported by $\{\tyj\}$, implying $\DDd(\cF)=1$.
Downwards closedness and finiteness of $\cF$ also implies
that $\DD(\cF)=|\Max(\cF)|$.  Clearly, the maximal sets in $\cF$
are just total functions $f\colon X\to Y$, so there are
$|Y|^{|X|}=n^\ell$ of them and $\DD(\cF)=n^\ell$.
Finally, the downwards closedness of $\cF$ implies that for any
such maximal~$f$, we have $\varjo_f(\cF)=\pot(f)=[\tyj,f]$, i.e.,
shadow are intervals.  Hence, $\CD(\cF)=\DD(\cF)$.

The family $\cX$ is obviously also downward closed, so we have
$\DDd(\cX)=1$ and $\CD(\cX)=\DD(\cX)$ in this case, too.
It is easy to see that the maximal set in $\cX$ are of form~$A\times B$
where $\{A,B\}$ is a partition of the set~$X$.
(In contrast, $\tyj=\tyj\times X=X\times\tyj\in\cX$ is not maximal,
as $\{(a,b)\}\in\cX$ for any distinct $a,b\in X$.)
The number of possible $A$'s, i.e., non-empty proper subsets of $X$
is indeed $2^n-2$.

In all the other cases, we have already determined dominating and
supporting families of the smallest sizes in the previous lemmas,
so the rest is simply combinatorial counting.
By \ref{dimcalc:incl}, items d~and e,
\[
  \DD(\cI_{\osaj})=|\Crit(\cI_{\osaj})\jer\{R_{\{a\}} \mid a\in X \}|
  =|\pot(X)|-|X|=2^\ell-\ell.
\]
By item~f, dual shadows are always intervals, so
$\CD(\cI_{\osaj})=\DDd(\cI_{\osaj})$. A combinatorial calculation
related to items~g and~h  gives the formula for $\DDd(\cI_{\osaj})$.
%
%

By lemma~\ref{dimcalc:anonym} item~b, $\{ A\times Y \mid A\osaj X\}$
is the unique smallest subfamily dominating $\cY$, and obviously
it is equipotent with $\pot(X)$, so $\DD(\cY)=2^\ell$.
By item~c, the set in the smallest family supporting $\cY$
are of the form $f\cup g$ where $f,g\colon A\to Y$ with $A\osaj X$
and for everu $x\in A$ we have $f(x)\neq g(x)$.  If the size $k=|A|$
is known, there are $\binom{\ell}{k}$ ways to choose $A$, and given
that $A$ and $x\in A$, there are $\binom{n}{2}$ ways to choose the pair
$\{f(x),g(x)\}$ (this is all that matters).  So for every $A$ with
size $k$, there are $\binom{n}{2}^k$ ways to choose $f\cup g$.
Summing this up for different sizes of $A$, we get the displayes
formula.  $\CD(\cY)=\DD(\cY)$, as dual shadows are intervals.

By \ref{dimcalc:indep} item~g, shadows are intervals, so
$\CD(\cI_{\perp})=\DD(\cI_{\perp})$.  Calculating the sizes of
critical and dual critical subfamilies (determined in item f)
with get the corresponding
formulas for upper dimension and dual upper dimension.
%
%
\end{proof}

There remains one interesting team-semantics-related
class of families of sets we need to investigate.
Let $X$, $Y$ and $Z$ be non-empty finite sets.  We shall consider
$$
\cI_{\perp,\bullet}
=\bigl\{\, \bigcup_{c\in Z} (A_c\times B_c\times\{c\})
           \;\bigm\vert\; \forall c\in Z\, (A_c\osaj X,\,B_c\osaj Y) \,\bigr\}.
$$

This time we will content ourselves on evaluating only lower and
upper bounds for this family instead of the exact values.
However, this is done within a more general framework which can be
applied to other similar cases.

\begin{definition} 
Families of sets $\cA$ and $\cB$ are called \emph{similar}
if there exists a bijection $f\colon X\to Y$ such that
$\cA\osaj\pot(X)$ and 
$$
  \cB = \{ f[A] \mid A\in\cA \}.
$$
\end{definition}

It is then straightforward to show that:

\begin{proposition} \label{dim-sim}
Let $\cA$, $\cB$ and $\cC$ be similar families of sets.  Then:
\begin{enumerate}[(a)]
\item
If $\cA$ and $\cB$ are similar, then  
$\DD(\cA)=\DD(\cB)$, $\DDd(\cA)=\DDd(\cB)$  and $\CD(\cA)=\CD(\cB)$.

\item
If $\cC=\cA\cap \pot(C)$ for some $C$, then
$\DD(\cC)\le\DD(\cA)$, $\DDd(\cC)\le\DDd(\cA)$  and $\CD(\cC)\le\CD(\cA)$.  
\end{enumerate}
\end{proposition}

\begin{definition}
Let $(\cA_i)_{i\in I}$ be an indexed family of families of sets.
Then its \emph{general tensor disjunction} is the family
$$
\gtd_{i\in I}\cA_i=\bigl\{\, \bigcup_{i\in I}A_i
     \;\bigm\vert\; \forall i\in I\,(A_i\in\cA) \,\bigr\}.
$$  
\end{definition}

\smallskip
Note that if the base sets of the families $\cA_i\osaj\pot(X_i)$
are all disjoint, i.e., if $(X_i)_{i\in I}$ is a disjoint family,
then there is a natural bijection $A\mapsto (A\cap X_i)_{i\in I}$
between $\gtd_{i\in I}\cA_i$ and $\prod_{i\in I}\cA_i$.
In the other end of the spectrum, if $\cA$ is closed under unions,
then $\gtd_{i\in I}\cA=\cA$.

\begin{proposition} \label{dim-gtd}
Let $(\cA_i)_{i\in I}$ be an indexed family of families of sets.
Then 
$$
  \CD\Bigl(\gtd_{i\in I}\cA_i\Bigr)\le\prod_{i\in I}\CD(A_i).
$$  
\end{proposition}

\begin{proof}
Pick, for each $i\in I$, an index set $J_i$ and intervals
$\cL_{i,j}$, $j\in J_i$ with $|J_i|=\CD(\cA_i)$ and
$\bigcup_{j\in J_i}\cL_{i,j}=\cA_i$.  Write $\cL_{i,j}=[B_{i,j},C_{i,j}]$.
For each $f\in J=\prod_{i\in I}J_i$, consider the interval
$\cL_f=[B_f,C_f]$ where
$$
  B_f = \bigcup_{i\in I} B_{i,f(i)}
  \text{ and }
  C_f = \bigcup_{i\in I} C_{i,f(i)}.
$$
Then clearly $|J|=\prod_{i\in I}|J_i|=\prod_{i\in I}\CD(\cA_i)$
and $\gtd_{i\in I}\cA_i=\bigcup_{\in I}\cL_f$.
\end{proof}

As a corollary, we get the desired estimates.
\begin{proposition}\label{lower_upper-bound}
Let  $X$, $Y$, and $Z$ be finite base sets
with $\ell=|X|\ge 2$, $n=|Y|\ge 2$ and $s=|Z|\ge 1$.
Then
\begin{align*}
  (2^\ell-\ell-1)(2^n-n-1)+1
  &\le\min\{\DD(\cI_{\perp,\bullet}),\DDd(\cI_{\perp,\bullet})\} \\
  &\le\CD(\cI_{\perp,\bullet})
  \le ((2^\ell-\ell-1)(2^n-n-1)+\ell+n)^s.
\end{align*}
\end{proposition}

\begin{proof}
For each $c\in Z$, put 
$$
   \cJ_c=\cI_{\perp,\bullet}\cap\pot(X\times Y\times\{c\})
   =\bigl\{\, A_c\times B_c\times\{c\}
           \;\bigm\vert\; A_c\osaj X,\,B_c\osaj Y) \,\bigr\}.
$$
Clearly, $\cJ_c$ is similar to $\cI_{\perp}$, so
by Theorem~\ref{dim-calculations} and Proposition~\ref{dim-sim}, case a, 
we have
$$
(2^\ell-\ell-1)(2^n-n-1)+1  = \min\{\DD(\cI_{\perp}),\DDd(\cI_{\perp})\}
=\min\{\DD(\cJ_c),\DDd(\cJ_c)\}.
$$
Since $\cJ_c=\cI_{\perp,\bullet}\cap\pot(X\times Y\times\{c\})$,
Propositions~\ref{dim-sim} and \ref{basic:dimEst}, further imply that
$$
\min\{\DD(\cJ_c),\DDd(\cJ_c)\}
\le \min\{\DD(\cI_{\perp,\bullet}),\DDd(\cI_{\perp,\bullet})\}
\le \CD(\cI_{\perp,\bullet}).
$$
It easy to see that $\gtd_{c\in Z}\cJ_c=\cI_{\perp,\bullet}$,
so now when we combine the results of Theorem~\ref{dim-calculations}
and Proposition~\ref{dim-gtd}, we get the inequality
$$
   \CD(\cI_{\perp,\bullet})=\CD(\gtd_{c\in Z}\cJ_c)
   \le\prod_{c\in Z} \CD(\cJ_c)=((2^\ell-\ell-1)(2^n-n-1)+\ell+n)^s.
$$
\end{proof}

In our logical application, when we apply Theorem~\ref{dim-calculations}
and the previous proposition to determine
the dimension functions of the corresponding atomic formulas, we shall
face a technical complication:  The dimension functions of formulas
depend on the set of variables that are interpreted in the teams
of assignments.  The previous result corresponds exactly to the situation
where only the variables occuring in the atomic formula are interpreted,
but there might be dummy variables to be considered. We shall need
the next proposition to overcome this difficulty: the effect of dummy
variables is not critical.  In this intended application,
the surjective function in the proposition
will be the restriction of the assignment to the occuring variables.

\begin{proposition} \label{calc:dummyabs}
Let $p\colon X\to Y$ be a surjection.  Recall that the inverse
projection is the operation
$\Delta_{p^{-1}}\colon\cP(\cP(Y))\to\cP(\cP(X))$,
\[
  \Delta_{p^{-1}}(\cY)=\{A\in\cP(X)\mid p[A]\in\cY\}.
\]
Suppose that $s,r\in\NN$ are constants such that
for each $y\in Y$, we have $|p^{-1}\{y\}|\le s$,
and for each $B\in\cB$, we have $|B|\le r$.
Let $\cB\osaj\pot(Y)$ and $\cA=\Delta_{p^{-1}}(\cB)$.
Then
\[
  \DD(\cA)=\DD(\cB),\;
  \DDd(\cA)\le s^r\DDd(\cB) \text{ and }
  \CD(\cA)\le s^r\CD(\cB).
\]
\end{proposition}

\begin{proof}
Choose a subfamily $\cG\osaj\cB$ such that $\cG$ dominates $\cB$
and  $\DD(\cB)=|\cG|$.  Now clearly $\cG'=\{p^{-1}[G]\mid G\in\cG\}$
dominates $\cA$, so $\DD(\cA)\le\DD(\cB)$.  If there were
a family $\cG''\osaj\cA$ dominating $\cA$ such that $|\cG''|<\DD(\cB)$,
then $\cG^*=p[[\cG'']]=\{p[A] \mid A\in\cG''\}$ would dominate $\cB$
contrary to the definition of the upper dimension.
Hence, $\DD(\cA)=\DD(\cB)$.

The cases of dual upper dimension and cylindrical dimension are
slightly more involved.  The point is that even if $L$ were minimal
in $\cB$, the inverse image $p^{-1}[L]$ is not in general minimal in $\cA$.
Call $A$ a \emph{selective inverse image}
of $B$, if $p[A]=B$
and $p\restriction A$ is an injection.  Note that $A$ is a selective
inverse image of $B$ if and only if $A$ is a minimal  set with $p[A]=B$.
Choose now $\cK$ that supports $\cB$
and $\DDd(\cB)=|\cK|$.  Consider the family $\cK'$ of all sets $A\in\cA$
such that $A$ is selective inverse image of some $B\in\cK$.
Clearly $\cK'$ supports $\cA$.  Each $B\in\cK$ has at most
$s^{|B|}\le s^r$ selective inverse images, as for every $b\in B$,
we have $|p^{-1}\{b\}|\le s$.  Hence, $\DDd(\cA)\le |\cK'|\le s^r\DD(\cB)$.
In the case of the cylindrical dimension, the proof is similar.
\end{proof}

\subsection{Dimensions of definable families}

We have defined three dimension concepts for totally arbitrary
families of sets on a finite base set. We now apply these concepts to
definable families of subsets of a cartesian product $M^m$. In
particular, we are interested in calculating the three dimensions for
families of the form $\sat{\phi}{{M,\vec x}}$.

\begin{lemma}
If $\phi$ is first order, then $\DD(\sat{\phi}{{M,\vec x}})=\DDd(\sat{\phi}{{M,\vec x}})=\CD(\sat{\phi}{{M,\vec x}})=1$.
\end{lemma}

\begin{proof}
The claim follows from the fact that, as we noted above, if $\phi(x_0,\ldots,x_{m-1})$ is first order, then $\sat{\phi}{{M,\vec x}}=[\emptyset,T_\phi]$. This makes the dimension computations trivial.\end{proof}

As alluded to in Section~\ref{families-of-teams}, team semantics permits the extension of first order logic  by a number of new atoms (see Definition~\ref{atoms}) leading to dependence logic (\cite{Vaa}), inclusion logic (\cite{Ga}), exclusion logic (\cite{Ga}), independence logic (\cite{Ga}), and anonymity logic (\cite{anon}). In order to estimate the dimensions of families definable in these logics we first note the following consequence of Theorem~\ref{dim-calculations}:

\begin{theorem}\label{dim-atoms}
Suppose $|M|=n$.
\begin{enumerate}[(a)] 
\item Let $\alpha$ be the dependence atom $\dep(\vx,y)$, where
  $\len(\vx)=m$, and let $\vz=\vx y$. Then
  $\DD(\sat{\alpha}{{M,\vz}})=\CD(\sat{\alpha}{{M,\vz}})=n^{n^m}$ and
  $\DDd(\sat{\alpha}{{M,\vz}})=1$.
\item Let $\alpha$ be the exclusion atom $\vx\mid\vy$, where
  $\len(\vx)=\len(\vy)=m$, and let $\vz=\vx \vy$. Then
  $\DD(\sat{\alpha}{{M,\vz}})=\CD(\sat{\alpha}{{M,\vz}})=2^{n^m}-2$
  and $\DDd(\sat{\alpha}{{M,\vz}})=1$.

\item Let $\alpha$ be the inclusion atom $\vx\subseteq\vy$, where
  $\len(\vx)=\len(\vy)=m$, and let $\vz=\vx \vy$. Then
  $\DD(\sat{\alpha}{{M,\vz}})=2^{n^m}-n^m$ and
  $\DDd(\sat{\alpha}{{M,\vz}})=\CD(\sat{\alpha}{{M,\vz}})
  =1+\sum^{n^m}_{k=2}\binom{n^m}{k}k^k$.
\item Let $\alpha$ be the anonymity atom $\vx\anonym y$, where
  $\len(\vx)=m$. Then
  $\DD(\sat{\alpha}{{M,\vz}})=\CD(\sat{\alpha}{{M,\vz}})=2^{n^m}$
  and $\DDd(\sat{\alpha}{{M,\vz}})
       =\sum^{n^m}_{k=0}\binom{n^m}{k}\binom{n}{2}^k $.
\item Let $\alpha$ be the pure independence atom $\vx\perp\vy$, where
  $\len(\vx)=m$ and $\len(\vy)=k$, and let $\vz=\vx \vy$. Then
  $\DD(\sat{\alpha}{{M,\vz}})=\CD(\sat{\alpha}{{M,\vz}})=
  (2^{n^m}-n^m-1)(2^{n^k}-n^k-1)+n^m+n^k$ and
  $\DDd(\sat{\alpha}{{M,\vz}})=(2^{n^m}-n^m-1)(2^{n^k}-n^k-1)+1$.
\item Let $\alpha$ be the conditional independence atom $ \vx\perp_{\vu} \vy$, where
  $\len(\vx)=m$, $\len(\vy)=k$, $\len(\vu)=s$, and let $\vz=\vx \vu\vy$. Then
  $(2^{n^m}-n^m-1)(2^{n^k}-n^k-1)+n^m+n^k\le\DD(\sat{\alpha}{{M,\vz}})\le\CD(\sat{\alpha}{{M,\vz}})\le
  ((2^{n^m}-n^m-1)(2^{n^k}-n^k-1)+n^m+n^k)^{n^s}$ and
  $(2^{n^m}-n^m-1)(2^{n^k}-n^k-1)+1\le\DDd(\sat{\alpha}{{M,\vz}})\le((2^{n^m}-n^m-1)(2^{n^k}-n^k-1)+1)^{n^s}$.

\end{enumerate}
\end{theorem}

\begin{proof}
(a) Letting $\cF = \{ f\osaj M^m\times M \mid f\text{ is a mapping }\}$, Theorem~\ref{dim-calculations} gives $\DD(\cF)=\CD(\cF)=n^{n^m}$ and $\DDd(\cF)=1$. By Definition~\ref{atoms} we have $\cF=\sat{\alpha}{{M,\vz}}$ and the claim follows. The short argument is the same in each other case (b)-(f). In (f) we use Proposition~\ref{lower_upper-bound}.
\end{proof}


\begin{table}
    \centering
    \begin{tabularx}{0.8\linewidth}{cXl}
atom & upper dimension & notes \\      
\hline
$x=y$&$1$&\\
$\dep(\vec{y})$&$n^m$&$\len(\vec{y})=m$\\
$\vec{x}\subseteq\vec{y}$&$2^{n^m}-n^m$&$\len(\vec{x})=\len(\vec{y})=m$\\
$\vec{x} \mid \vec{y}$ & $2^{n^m}-2$ & $\len(\vec{x})=\len(\vec{y})=m$\\
$\vec{x}\anonym {y}$ & $2^{n^m}$ & $\len(\vec{x})=m$ \\
$\vec{x}\perp\vec{y} $ & $\approx 2^{n^m+n^k}$
     & $\len(\vec{x})=m, \len(\vec{y})=k$\\
$\dep(\vec{x},{y})$ & $n^{n^m}$ & $\len(\vec{x})=m$ \\
$\vec{x}\perp_{\vec{u}}\vec{y}$ & $\approx [2^{n^{m}+n^{k}},2^{n^{m+s}+n^{k+s}}]$ & $\len(\vec{x})=m, \len(\vec{y})=k, \len(\vec{u})=s$ \\
    \end{tabularx}
    \caption{Upper dimensions of atoms.}
    \label{tab:my_label}
\end{table}

We may notice that, keeping $m$ and $k$ fixed, the upper and the
cylindrical dimension of the dependence atom grows faster than the
respective dimensions of the other atoms, except the relativized independence atom. Varying $m$ and $k$ we
obtain a host of comparisons between dimensions of the atoms. These
will become relevant below when we combine the atoms with logical
operations.

 Let us now define the important concept of locality:

\begin{definition}\label{locality}
A formula $\phi$ of any logic, with the free variables $\vx$, is said to be \emph{local} if for all models $M$ and teams $T$ with $\vx\subseteq\dom(T)$ we have $$M\models_T\phi\ \Leftrightarrow\  M\models_{T\restriction\vx}\phi.$$
\end{definition}

All the atoms of Definition~\ref{atoms} are local and the logical operations of Definition~\ref{fol}, as well as all Lindstr\"om quantifiers (see Definition~\ref{Linquantsem}), preserve locality. 

The semantics defined in Definition~\ref{local} has  a variant called \emph{strict semantics}. In strict semantics we define the meaning of tensor disjunction by 
$M\models_T\phi\vee\psi$ if and only if $T=Y\cup Z$ such that
    $M\models_Y\phi$, $M\models_Z\psi$, and $Y\cap Z=\emptyset$. The meaning of existential quantifier in strict semantics is 
   $M\models_T\exists x\phi$ if and only if there is
  $F:T\to M$ such that
  $M\models_{T[F/x]}\phi$. For dependence logic this change of semantics does have no effect because of downward closure. However, inclusion logic with strict semantics is not local. We will not consider  strict semantics in detail in this paper.

It is also important to notice that above, we have calculated
the dimensions of the teams related to certain atomic formulas
relative to the variables occurring in the formulas.
In general, we need to consider atomic formulas -- or, in general,
also other formulas -- as subformulas of larger formulas, so
we need to attach also other variables in the context.
Usually, the following estimates 
are good enough for our purposes.

\begin{proposition} \label{calc:dummydf}
Let $\mM$ be a structure and $\phi$ a local formula with a common vocabulary,
$\vy$ the sequence of variables occurring in $\phi$  and $\vx$
a finite sequence of variables extending~$\vy$. Suppose $\mM$
has size $n$, $r$ is constant such that for every team $T$ in variables $\vy$
we have that $\mM\models_T\phi$ implies $|T|\le r$,
and $t=\len(\vx)-\len(\vy)$.
Then
\[
  \DD(\sat{\phi}{\mM,\vx})=\DD(\sat{\phi}{\mM,\vy}),\,
  \DDd(\sat{\phi}{\mM,\vx})\le n^{tr}\DD(\sat{\phi}{\mM,\vy})\text{ and }
  \CD(\sat{\phi}{\mM,\vx})\le n^{tr}\CD(\sat{\phi}{\mM,\vy}).
\]
\end{proposition}

\begin{proof}
This is a simple application of the Proposition~\ref{calc:dummyabs}.
Put $k=\len(\vy)$.
Consider the case where $T=M^{k+t}$, $Y=M^k$, $p\colon T\to Y$
is the natural projection, $\cA=\sat{\phi}{\mM,\vx}$ and
$\cB=\sat{\phi}{\mM,\vy}$.  The locality of $\phi$ implies that
$\cA=\Delta_{p^{-1}}(\cB)$, and for each $y\in Y$,
$|p^{-1}\{y\}|=n^t$, and for every $T\in\cB$, it holds that
$|T|\le r$.  As $|Y|=n^k$, the results follow. 
\end{proof}

\section{Growth classes and operators} \label{growth}

Although the basic dimension concepts above apply perfectly to any family of sets, we can say more when we focus on families of subsets of cartesian powers of finite sets i.e. families of teams. In such a framework the concept of a growth class arises naturally and is the topic of this section.

\subsection{Growth classes}\label{gc}

As we apply our dimensional techniques to definability problems
on the class of finite structures, we are constantly facing the dilemma
that it is usually not sufficient to consider a single structure
and families of sets arising from team semantics in that structure,
but we rather have to consider the class of all appropriate
finite structures.  That means that we have to accept the possibility
that the size of the base set may change, which calls for a dynamical
way to handle matters.  To that end, we consider growth classes.

In the definitions that follow, we generalize the arithmetical notation
in the pointwise fashion, e.g., for functions $f,g\colon \NN\to\NN$
we set $f+g$ to be the function $\NN\to\NN$ such that
$(f+g)(n)=f(n)+g(n)$, for $n\in\NN$, and $f\le g$ means that
$f(n)\le g(n)$ holds for every $n\in\NN$.

\begin{definition}
A set $\OO$ of mappings $f\colon\NN\to\NN$ is a \emph{growth class}
if the following conditions hold for all $f,g\colon\NN\to\NN$:
\begin{enumerate}[(a)]
\item
If $g\in\OO$ and $f\le g$, then $f\in\OO$.
\item
If $f,g\in\OO$, then $f+g\in\OO$ and $fg\in\OO$.  
\end{enumerate}
\end{definition}

The point of growth classes is that they are closed under natural operators arising from logical operations. As it turns out, if we figure out the growth classes of some atoms, anything definable from those atoms by means of most of the logical operations we deal with will be in the same growth class. Thus the growth classes represent important dividing lines.

We are interested in the following particular classes:
For $k\in\NN$, the class $\EE_k$ consist all $f\colon\NN\to\NN$
such that there exists a polynomial $p\colon\NN\to\NN$ of degree $k$ and
with coefficients in $\NN$ such that for all $n\in\NN$ $$f(n)\le 2^{p(n)}.$$  In addition,
$\FF_k$ is the class of functions $f\colon\NN\to\NN$
such that there exists a polynomial $p\colon\NN\to\NN$ of degree $k$ and
with coefficients in $\NN$ such that for every $n\in\NN\jer\{0,1\}$
we have that 
\[
  f(n)\le n^{p(n)}. 
\]

Note that $\EE_0$ is the class of bounded functions and $\FF_0$
the class of functions of polynomial growth.
The following is immediate:
\begin{proposition}
Each $\EE_k$ and $\FF_k$ (for $k\in\NN$) is a growth class.
Furtheoremore, we have that
\[
  \EE_0\aioj\FF_0\aioj\EE_1\aioj\FF_1\aioj\cdots\aioj\EE_k\aioj\FF_k\cdot
\]
\end{proposition}

\begin{definition}
To each formula $\phi$ with free variables in $\vx$
allowing a team-semantical interpretation we relate
the following \emph{dimension functions}:

\begin{align*}
&\Dim_{\phi,\vec x}\colon\NN\to\Card,
  &\Dim_{\phi,\vec x}(n)&=\sup\left\{\DD(\sat{\phi}{{M,\vx}}) \mid
                \text{$M$ is a model}, |M|=n\right\},\\  
&\Dimd_{\phi,\vec x}\colon\NN\to\Card,
  &\Dimd_{\phi,\vec x}(n)&=\sup\left\{\DDd(\sat{\phi}{{M,\vx}}) \mid
                \text{$M$ is a model}, |M|=n\right\},\\  
&\CDim_{\phi,\vec x}\colon\NN\to\Card,
  &\CDim_{\phi,\vec x}(n)&=\sup\left\{\CD(\sat{\phi}{{M,\vx}}) \mid
                \text{$M$ is a model}, |M|=n\right\}.\\  
\end{align*}
\end{definition}

\begin{example}\label{atomiendimensioita}
\begin{enumerate}[(a)]
    \item $\CDim_{\phi,\vec x}(n)=1$, hence $\CDim_{\phi,\vec x}$ is in $\EE_0$, for every first order $\phi$.  Hence the same holds for $\Dim_{\phi,\vec x}$ and $\Dimd_{\phi,\vec x}$, by Proposition~\ref{basic:dimEst}.
    \item $\Dim_{=\!(\vx,y),\vx y}(n)=n^{n^k}$,  hence $\Dim_{=\!(\vx,y),\vx y}$ is in $ \FF_k$, where $\len(\vx)=k$. The same holds for   $\CDim_{=\!(\vx,y),\vx y}$. However, $\Dimd_{=\!(\vx,y),\vx y}(n)=1$,  whence $\Dimd_{=\!(\vx,y),\vec x y}$ is in $ \EE_0$.
    \item $\Dim_{\vx|\vy,\vx\vy}(n)=2^{n^k}-2$,  hence $\Dim_{\vx|\vy,\vx\vy}$ is in $ \EE_k$, where  $\len(\vx)=\len(\vy)=k$. The same holds for   
$\CDim_{\vx|\vy,\vx\vy}$. However, $\Dimd_{\vx|\vy,\vx \vy}(n)=1$,  whence $\Dimd_{\vx|\vy,\vx\vy}$ is in $ \EE_0$.
   \item $\Dim_{\vx\subseteq\vy,\vx\vy}(n)=2^{n^k}-n^k$,  hence $\Dim_{\vx\subseteq\vy,\vx\vy}$ is in $ \EE_k$, where $\len(\vx)=\len(\vy)=k$. \item $\Dim_{\vx\Upsilon y,\vx y}(n)=2^{n^k}$,  hence $\Dim_{\vx\Upsilon y,\vx y}\in \EE_k$, where $\len(\vx)=k$.
\item $\Dim_{\vx\perp_{\vz}\vy,\vx\vz\vy}(n)\in [r,r^{n^s}]$, where $r=(2^{n^m}-n^m-1)(2^{n^k}-n^k-1)+n^m+n^k$,  hence $\Dim_{\vx\perp_{\vz}\vy,\vx\vz\vy}$ is in $ \EE_{m+k+s}$, where $\len(\vx)=k$, $\len(\vy)=m$, and $\len(\vz)=s$. 

 \end{enumerate}
\end{example}

For a summary of the above example, see Table~\ref{classes}. Note that the last row of the table indicates an upper bound only.

\begin{table}[h]
\begin{center}
\def\arraystretch{1.2}\begin{tabularx}{0.65\linewidth}{ccccclll}
family & $X$ & $Y$ & $Z$&atom $\alpha$
           & $\Dim_\alpha$ & $\Dimd_\alpha$ & $\CDim_\alpha$\\
\hline
$\cF$ & $M^k$ & $M$ && $\dep(\vx,t)$
           & $\FF_k$ & $\EE_0$ & $\FF_k$ \\
$\cX$ & $M^k$ & $M^k$ && $\vx\mid\vy$
           & $\EE_k$ & $\EE_0$ & $\EE_k$ \\
$\cI_{\osaj}$ & $M^k$ & $M^k$ && $\vx\osaj\vy$
           & $\EE_k$ & $\FF_k$ & $\FF_k$ \\
$\cY$ & $M^k$ & $M$ & &$\vx\anonym  y$
           & $\EE_k$ &$\FF_0$ & $\EE_k$ \\
           $\cI_{\perp}$ & $M^k$ & $M^l$ && $\vx\perp\vz$
           & $\EE_{k+l}$ & $\EE_{k+l}$ & $\EE_{k+l}$ \\
$\cI_{\perp,\bullet}$ & $M^k$ & $M^l$ &$M^s$& $\vx\perp_{\vz}\vy$
           & $\EE_{k+l+s}$ & $\EE_{k+l+s}$ &$\EE_{k+l+s}$  \\
\end{tabularx}\def\arraystretch{1}
\end{center}
\caption{\label{classes}Growth classes of  families arising from atoms.}
\end{table}

In the example above, the growth classes of the dimension functions of
some atoms were determined relative to variables occurring in the formula.
In the general case, it is conceivable that the dimensions functions
are not preserved in the same classes.  We need the following concept
to show that the situation is, by and large, conserved.

\begin{definition}
A formula $\phi$ with free variables $\vx$ is of \emph{degree}~$k\in\NN$
if there is a polynomial function $p\colon\NN\to\NN$ of degree~$k$,
with coefficients in $\NN$, such that the following holds:
For every structure $\mM$ for the common vocabulary of size $n\in\NN$,
if $\mM\models_{T}\phi$ holds for a team in variables~$\vx$,
then $|T|\le p(n)$.
\end{definition}

For a local formula with $k$ free variables the degree is always at most $k$.


\begin{proposition}
  Let $l\in\NN$, $\OO$ be a growth class,
  $\phi$ be a local formula of degree $k$,
  $\vy$ be the tuple of variables
  occurring in $\phi$, and $\vx$ be a finite tuple extending $\vy$.
  \begin{enumerate}[(a)]
  \item If $\Dim_{\phi,\vy}$ is in $\OO$,
        then $\Dim_{\phi,\vx}$ is also in $\OO$.
  \item If $\FF_k\osaj\OO$ and $\Dimd_{\phi,\vy}$ is in $\OO$, 
        then $\Dimd_{\phi,\vx}$ is also in $\OO$.
  \item If $\FF_k\osaj\OO$ and $\CDim_{\phi,\vy}$ is in $\OO$, 
        then $\CDim_{\phi,\vx}$ is also in $\OO$.
  \end{enumerate}
\end{proposition}

\begin{proof}
The proof is a direct application of Proposition~\ref{calc:dummydf}.
Fix the polynomial function~$p$ of degree~$k$ witnessing that
$\phi$ is of degree~$k$, and put $t=\len(\vx)-\len(\vy)$.
Consider an appopriate structure $\mM$ of size~$n$.
By the Proposition (putting $r=p(n)$), we have
\begin{align*}
  &\DD(\sat{\phi}{\mM,\vx})=\DD(\sat{\phi}{\mM,\vy})\\
  &\DDd(\sat{\phi}{\mM,\vx})\le n^{tp(n)}\DDd(\sat{\phi}{\mM,\vy})\text{ and }\\
  &\CD(\sat{\phi}{\mM,\vx})\le n^{tp(n)}\CD(\sat{\phi}{\mM,\vy}).\\
\end{align*}
As $tp$ is a polynomial function of degree~$k$, the function
$n\mapsto n^{tp(n)}$ is in $\FF_k$, and the results follow.
\end{proof}

It is worth noting that dual and cylindrical dimensions of formulas do not behave as well as upper dimension when new variables are added (see Theorem~\ref{calc:dummydf}).
Thus the dual or cylindrical dimension of a formula may be in $\EE_k$, but when new variables are taken into account, even if they do not occur in the formula, the (dual or cylindrical) dimension may jump into $\FF_k$ as a a new factor $n^{n^k}$ may appear.

\subsection{Kripke-operators}

Our goal in this section is to find natural criteria for operators to preserve growth classes. 
We start by defining a class of operators that is inspired by the Kripke semantics of modal logic. Let $X$ and $Y$ be nonempty base sets, and let
$\cR\subseteq \cP(Y)\times\cP(X)^n$ be an $(n+1)$-ary relation. Then we define a corresponding operator $\Delta_\cR\colon \cP(\cP(X))^n\to\cP(\cP(Y))$
by the condition
\begin{align*}
   B\in\Delta_\cR(\cA_0,\ldots,\cA_{n-1})
	\Leftrightarrow \exists A_0\in \cA_0\ldots
	\exists A_{n-1}\in \cA_{n-1}:
	\; (A,A_0,\ldots,A_{n-1})\in \cR.
\end{align*}
Note that $\Delta_\cR$ can be seen as the $n$-ary second-order version of the function mapping the truth set of a formula $\varphi$ to the truth set of $\Diamond\varphi$ in a Kripke model.

\begin{definition}
Let $X$ and $Y$ be nonempty sets. A function  $\Delta\colon \cP(\cP(X))^n\to\cP(\cP(Y))$
is a \emph{(second-order) Kripke-operator\footnote{This notion is
defined by  \cite{Lu}; he calls functions satisfying the
condition just ``operators''.}}, if there is a relation
$\cR\subseteq \cP(Y)\times\cP(X)^{n}$ such that $\Delta=\Delta_\cR$.
\end{definition}

In the next example we go through the operators introduced in Example~\ref{operator-ex}, and check
which of them are Kripke-operators.

\begin{example}\label{kripke-ex} 
\begin{enumerate}[(a)]
\item 
Intersection of families is a Kripke-operator on any base set
$X$: If $\cA,\cB\subseteq\cP(X)$ and $C\in \cP(X)$, then
$C\in\cA\cap\cB$ if and only if there exist $A\in\cA$ and $B\in\cB$
such that $(C,A,B)\in\cR_\cap$, where $\cR_\cap$ is the simple
relation $\{(D,D,D)\mid D\in\cP(X)\}$.

\item
Union of families on $X$ is \emph{not} a Kripke-operator. This
is because for any relation $\cR\subseteq (\cP(X))^3$ and any
nonempty family $\cA\subseteq\cP(X)$ we have
$\Delta_\cR(\cA,\emptyset)=\emptyset\not=\cA=\cA\cup\emptyset$. However,
defining
$\cR_{\cup^*}=\{(A,A,\emptyset),(A,\emptyset,A)\mid A\in\cP(X)\}$ we
obtain a Kripke-operator $\Delta_{\cR_{\cup^*}}$ that captures union
when restricting to families that contain $\emptyset$.

\item
It is also easy to see that complementation $\Delta^X_c$ is not
a Kripke-operator: $\Delta_\cR(\emptyset)=\emptyset$ for any
relation $\cR\subseteq (\cP(X))^2$, but
$\Delta^X_c(\emptyset)=\cP(X)\not=\emptyset$.


\item
 Tensor disjunction  and tensor negation on $X$ are Kripke-operators: clearly $\cA\lor\cB=\Delta_{\cR_\lor}(\cA,\cB)$ and $\Delta^X_\lnot(\cA)=\Delta_{\cR_\lnot}(\cA)$ where $\cR_\lor=\{(A\cup B,A,B)\mid A,B\in\cP(X)\}$ and $\cR_\lnot=\{(X\jer A,A)\mid A\in\cP(X)\}$. More generally, for any binary operation $\oast$ on the set $\{0,1\}$ the corresponding tensor operator is a Kripke-operator:  $\Delta^X_{\oast}=\Delta_{\cR_\oast}$, where $\cR_\oast=\{(A\ast B,A,B)\mid A,B\in\cP(X)\}$ (see Definition \ref{tensor-def}).

\item
Projections and inverse projections are
Kripke-operators. Indeed, if $f\colon X\to Y$ is a surjection, then
clearly $\Delta_f=\Delta_{\cR_f}$, where
$\cR_f=\{(f[A],A)\mid A\in\cP(X)\}$. Similarly,
$\Delta_{f^{-1}}=\Delta_{\cR_{f^{-1}}}$, where
$\cR_{f^{-1}}=\{(A,f[A])\mid A\in\cP(X)\}$.


\item 
Finally we observe that the $(\cK,\vec\ell)$-projection operators $\Delta^{M^m}_{\cK,\vec\ell}$ corresponding to Lindström quantifiers are Kripke-operators: by Definition \ref{Q-oper} we have $\Delta^{M^m}_{\cK,\vec\ell}=\Delta_{\cR_{\cK,\vec\ell}}\,$, where $\cR_{\cK,\vec\ell}=\{(B,A)\mid B=\pi^p_{\cK,\vec\ell}\,(A)\}$ if $(M,\emptyset)\notin\cK$, and $\cR_{\cK,\vec\ell}=\{(B,A)\mid \pi^p_{\cK,\vec\ell}\,(A)\subseteq B\}$ if $(M,\emptyset)\in\cK$.

In particular, the existential quantification operators $\Delta^{M^m}_{\exists i}$ and the universal quantification operators $\Delta^{M^m}_{\forall i}$ are Kripke-operators.
\end{enumerate}

\end{example}

An important property of Kripke-operators is that they preserve unions of families:

\begin{lemma}[Union Lemma]\label{union-lem}
Let $\Delta_\cR\colon \pot(\pot(X))^n\to\pot(\pot(Y))$ be a Kripke-operator, and let 
$\cA_i^k\in\pot(\pot(X))$, $k\in K_i$, be families of sets for some index sets $K_i$, $i<n$.
Then 
$$
	\Delta_\cR\bigr(\bigcup_{k\in K_0}\cA_0^k,\ldots,\bigcup_{k\in K_{n-1}}\cA_{n-1}^k\bigl)
	= \bigcup_{\vec k\in K}
	\Delta_\cR(\cA_0^{k_0},\ldots,\cA_{n-1}^{k_{n-1}}),
$$
where we use the notation $\vec k=(k_0,\ldots,k_{n-1})$ and $K=K_0\times\cdots\times K_{n-1}$.
\end{lemma}

\begin{proof}
Using the notation $\cA_i=\bigcup_{k\in K_i}\cA_i^k$ for $i<n$ the left hand side of the equation can be written as $\cA:=\Delta_\cR(\cA_0,\ldots,\cA_{n-1})$. The claim follows now from the chain of equivalences below:
$$
\begin{array}{lcl}
	A\in\cA&\Leftrightarrow     &\forall i<n\,\exists A_i\in\cA_i: (A,A_0,\ldots,A_{n-1})\in\cR \\
	&\Leftrightarrow    &\forall i<n\,\exists k_i\in K_i\,\exists A_i\in\cA_i^{k_i}:
	(A,A_0,\ldots,A_{n-1})\in\cR \\
	&\Leftrightarrow    &\exists\vec k\in K: A\in\Delta_\cR(\cA_0^{k_0},\ldots,\cA_{n-1}^{k_{n-1}})\\
	&\Leftrightarrow    &A\in\bigcup_{\vec k\in K}\Delta_\cR(\cA_0^{k_0},\ldots,\cA_{n-1}^{k_{n-1}}).
	\end{array}
 $$

\end{proof}

Kripke-operators that preserve the property of being dominated (and/or supported) and convex have a crucial role in our considerations. This is because for such an operator $\Delta_\cR$ the (corresponding) dimension of the image $\Delta_\cR(\cA_0,\ldots,\cA_{n-1})$ is at most the product of the dimensions of $\cA_i$, $i<n$, and consequently, $\Delta_\cR$ preserves growth classes.

\begin{definition}
Let $\Delta\colon \cP(\cP(X))^n\to\cP(\cP(Y))$ be an operator. 
We say that $\Delta$ \emph{weakly preserves dominated (supported, resp.) convexity} if 
$\Delta(\cA_0,\ldots,\cA_{n-1})$ is dominated (supported, resp.) and convex or 
$\Delta(\cA_0,\ldots,\cA_{n-1})=\emptyset$ whenever 
$\cA_i$ is dominated and convex for each $i<n$. 
Furtheoremore, we say that $\Delta$ \emph{weakly preserves intervals} if $\Delta(\cA_0,\ldots,\cA_{n-1})$ is an interval or 
$\Delta(\cA_0,\ldots,\cA_{n-1})=\emptyset$ whenever 
$\cA_i$ is an interval for each $i<n$.
\end{definition}

\begin{example}\label{tensorintpres}
\begin{enumerate}[(a)]
\item
Proposition~\ref{basic:tensor-intpres} shows that each
tensor operator $\oast$ weakly preserves intervals.
(In this case, if $\cA$ and $\cB$ are nonempty, then so is $\cA\oast\cB$,
too, so we could blatantly state that $\oast$ preserves intervals,
dropping the specifier ``weakly''.)

\item
Suppose now the binary operation $\oast$ on the set $\{0,1\}$
is not monotone.  Recall that monotonicity of $\oast$ means that
for all $a,a',b,b'\in\{0,1\}$, whenever $a\le a'$ and $b\le b'$,
then $a\oast b\le a'\oast b'$ where $\le$ is the natural ordering
of the truth values with $0<1$.  10 of the 16 operations are not monotone,
i.e., all apart from the constant operations, projections, conjunction
and disjunction.  As $\oast$ is not monotone, there is $c\in\{0,1\}$
with
\[
\text{either }
  \begin{cases}
    c\oast 0=1\\
    c\oast 1=0\\ 
  \end{cases}
\text{ or }
  \begin{cases}
    0\oast c=1\\
    1\oast c=0.\\ 
  \end{cases}
\]
By symmetry, assume the former pair of equations. 
Consider now any $\cA\osaj\pot(X)$, and choose $C=\tyj$ if $c=0$,
and $C=X$ if $c=1$.  Note that $\{C\}=[C,C]\osaj\pot(X)$ is an
interval, so it is both dominated convex and supported convex.
Now
\[
  \{C\}\oast\cA = \{ C*A \mid A\in\cA \}
  = \{ X\jer A \mid A\in\cA \} =\lnot\cA.
\]
Picking any $\cA$ that is nonempty, dominated convex, but not
supported convex, we see that $\{C\}\oast\cA=\lnot\cA$ is nonempty, but
not dominated.  Thus, $\oast$ does not weakly preserve dominated
convexity. Similarly, interchanging the roles of ``dominated'' and ``supported''
we get that $\oast$ does not weakly preserve supported
convexity. 
\end{enumerate}

\end{example}

\begin{theorem}\label{sailyminen}
Let $\Delta_\cR\colon \cP(\cP(X))^n\to\cP(\cP(Y))$ be a Kripke-operator, and
let $\cA=\Delta(\cA_0,\ldots,\cA_{n-1})$.

\begin{enumerate}[(a)]
\item 
If $\Delta$ weakly preserves dominated convexity then 
$\DD(\cA)\le \DD(\cA_0)\cdot\ldots\cdot \DD(\cA_{n-1})$.

\item
If $\Delta$ weakly preserves supported convexity then 
$\DDd(\cA)\le \DDd(\cA_0)\cdot\ldots\cdot \DDd(\cA_{n-1})$.

\item
If $\Delta$ weakly preserves intervals
then $\CD(\cA)\le \CD(\cA_0)\cdot\ldots\cdot \CD(\cA_{n-1})$.
\end{enumerate}
\end{theorem}

\begin{proof}
(a) By Definition \ref{dimension-def}, for each $i<n$ there are dominated and convex subfamilies 
$\cA_i^{k}\subseteq\cA_i$, $k\in K_i$, such that $\cA_i=\bigcup_{k\in K_i}\cA_i^k$ and $|K_i|=\DD(\cA_i)$.
For each tuple $\vec k:=(k_0,\ldots,k_{n-1})$ in $K:=K_0\times\cdots\times K_{n-1}$, let
$\cA_{\vec k}$ denote the family $\Delta(\cA_0^{k_0},\ldots,\cA_{n-1}^{k_{n-1}})$.
By our assumption, each $\cA_{\vec k}$ is either dominated and convex, or empty.
By Lemma \ref{union-lem}, $\cA=\bigcup_{\vec k\in K}\cA_{\vec k}$. Thus we see that
$\DD(\cA)\le |K|=|K_0|\cdot\ldots\cdot |K_{n-1}|=\DD(\cA_0)\cdot\ldots\cdot \DD(\cA_{n-1})$.

Claim (b) is proved in the same way just by replacing dominated convexity by supported convexity.
Finally, to prove (c) it suffices to observe that a non-empty family is an interval if and only if it is dominated,
supported and convex. 
\end{proof}

As seen above in Example \ref{kripke-ex}, there are well-behaved operators that are not Kripke-operators, but on the other hand, most of the operators arising in our applications are Kripke-operators. Moreover, we can prove relatively simple exact characterizations for weak preservation of dominated convexity and supported convexity for Kripke-operators. 

Below we will use the notation 
$$
    \cR[A]:= \{(A_0,\ldots,A_{n-1})\mid (A,A_0,\ldots,A_{n-1})\in \cR\}.
$$

\begin{lemma}\label{domcon}
Let $\Delta_\cR\colon \cP(\cP(X))^n\to\cP(\cP(Y))$ be a Kripke-operator for finite $X$ and $Y$. 
Then $\Delta_\cR$ weakly preserves dominated convexity if and only if the following condition holds:
\begin{itemize}
\item[($*^\sharp$)] If $(A_0,\ldots,A_{n-1})\in \cR[A]$, $(B_0,\ldots,B_{n-1})\in \cR[B]$, and 
$C\in \DCH(A,B)$,
then there are 
$C_0,\ldots,C_{n-1}$ such that
$(C_0,\ldots,C_{n-1})\in \cR[C]$ and $C_i\in \DCH(A_i,B_i)$ for each $i<n$.
\end{itemize}
\end{lemma}

\begin{proof}
Assume that ($*^\sharp$) holds. Let $\cA_i$ be dominated convex sets with maximum sets
$D_i$ for $i<n$. If $\Delta_\cR(\cA_0,\ldots,\cA_{n-1})\not=\emptyset$, it contains maximal sets.
We show first that it has a unique maximal set.

Thus, assume that $A$ and $B$ are maximal sets in $\Delta_\cR(\cA_0,\ldots,\cA_{n-1})$.
Then there are $A_i,B_i\in\cA_i$, $i<n$, such that $(A_0,\ldots,A_{n-1})\in \cR[A]$ and
$(B_0,\ldots,B_{n-1})\in \cR[B]$. %
By ($*^\sharp$), there are $C_i$ such that
$(C_0,\ldots,C_{n-1})\in \cR[A\cup B]$ and $C_i\in \DCH(A_i,B_i)
\subseteq\cA_i$ for each $i<n$.
Hence $A\cup B\in \Delta_\cR(\cA_0,\ldots,\cA_{n-1})$. Since
$A,B\subseteq A\cup B$, this is possible only if $A=B$.

To prove that $\Delta_\cR(\cA_0,\ldots,\cA_{n-1})$ is convex, assume that  $A\subseteq C\subseteq B$, $(A_0,\ldots,A_{n-1})\in \cR[A]$ and 
$(B_0,\ldots,B_{n-1})\in \cR[B]$. Then $C\in [A,B]=\DCH(A,B)$,
whence by ($*^\sharp$), there are $C_i$ such that $C_i\in \DCH(A_i,B_i)\subseteq\cA_i$ for each $i<n$ and
$(C_0,\ldots,C_{n-1})\in \cR[C]$. Thus, $C\in \Delta_\cR(\cA_0,\ldots,\cA_{n-1})$.

Assume then that $\Delta_\cR$ weakly preserves dominated convexity. Let $(A_0,\ldots,A_{n-1})\in \cR[A]$, 
$(B_0,\ldots,B_{n-1})\in \cR[B]$, and $C\in \DCH(A,B)$. Since the families 
$\cA_i:=\DCH(A_i,B_i)$, $i<n$, are convex and dominated, there is a set
$D\in\Delta_\cR(\cA_0,\ldots, \cA_{n-1})$ such that $A\cup B\subseteq D$. Now $C\in \DCH(A,B,D)$,
and since $\Delta_\cR$ weakly preserves dominated convexity, $\DCH(A,B,D)\subseteq\Delta_\cR(\cA_0,\ldots,\cA_{n-1})$.  Thus ($*^\sharp$) holds. 
\end{proof}

\begin{lemma}\label{supcon}
Let $\Delta_\cR\colon \cP(\cP(X))^n\to\cP(\cP(Y))$ be a Kripke-operator for finite $X$ and $Y$. 
Then $\Delta_\cR$ weakly preserves supported convexity if and only if the following condition holds:
\begin{itemize}
\item[($*^\flat$)] If $(A_0,\ldots,A_{n-1})\in \cR[A]$, $(B_0,\ldots,B_{n-1})\in \cR[B]$, and
$C\in \SCH(A,B)$,  then there are 
$C_0,\ldots,C_{n-1}$ such that
$(C_0,\ldots,C_{n-1})\in \cR[C]$ and $C_i\in \SCH(A_i,B_i)$ for each $i<n$.
\end{itemize}
\end{lemma}

\begin{proof} The claim is proved in the same way as in the previous result.
\end{proof}

\subsection{Local Kripke-operators}

Many natural Kripke-operators $\Delta_\cR$ are local in the sense that the relation $\cR[A]$ is completely determined by its behaviour on singletons $\{a\}\subseteq A$. 

\begin{definition}\label{local}
A Kripke-operator $\Delta_\cR\colon \cP(\cP(X))^n\to\cP(\cP(Y))$ is
\emph{local}\footnote{This should not be confused with the concept
of locality for formulas. In \cite{Lu} this notion is defined
under the name ``transversal''.} if, for any $A\in\cP(Y)$,
$\cR[A]$ is determined by the relations $\cR[\{a\}]$, $a\in A$, as follows:
$(A_0,\ldots,A_{n-1})\in \cR[A]\Leftrightarrow$ for each $a\in A$
there is $(A^a_0,\ldots,A^a_{n-1})\in \cR[\{a\}]$ such that
$A_i=\bigcup_{a\in A}A_i^a$ for $i<n$.
\end{definition}

Lück proved (\cite{Lu}) that all local Kripke-operators $\Delta$
preserve flatness: if $\cA_i$, $i<n$, are flat (i.e., dominated and
downward closed), then $\Delta(\cA_0,\ldots,\cA_{n-1})$ is also
flat. We generalize this result to dominated convexity.

\begin{theorem}\label{LocalDomConPres}
If $\Delta_\cR\colon \cP(\cP(X))^n\to\cP(\cP(Y))$ is a local Kripke-operator for finite $X$ and $Y$, then it weakly preserves
dominated convexity.
\end{theorem}

\begin{proof}
It suffices to show that $\cR$ satisfies the condition
$(*^\sharp)$ of Lemma \ref{domcon}. Assume for this that $(A_0,\ldots,A_{n-1})\in \cR[A]$, $(B_0,\ldots,B_{n-1})\in \cR[B]$, and $C\in\DCH(A,B)=[A,A\cup B]\cup[B,A\cup B]$. We assume that
$C\in [A,A\cup B]$; the case $C\in [B,A\cup B]$ is similar. 

Since $\Delta_\cR$ is local,  for each $a\in A$ there are sets $A_i^a$
such that $A_i=\bigcup_{a\in A}A_i^a$ for $i<n$,
and $(A_0^a,\ldots, A_{n-1}^a)\in \cR[\{a\}]$. Similarly, for each $b\in B$ there are sets $B_i^b$
such that $B_i=\bigcup_{b\in B}B_i^b$ for $i<n$,
and $(B_0^b,\ldots, B_{n-1}^b)\in \cR[\{b\}]$.

Now, for each $i<n$, we define $C_i^c:=A_i^c$ for all $c\in A$, and $C_i^c:=B_i^c$ for all $c\in C\jer A$
(note that $A\subseteq C$ and $C\jer A\subseteq B$). Let $C_i:=\bigcup_{c\in C}C_i^c$ for $i<n$. By Definition
\ref{local} we have $(C_0,\ldots,C_{n-1})\in\cR[C]$.

We still need to show that $C_i\in \DCH(A_i,B_i)$ for $i<n$. Clearly
$C_i^c\subseteq A_i\cup B_i$ for each $c\in C$, whence $C_i\subseteq A_i\cup B_i$ for $i<n$. Furtheoremore, 
$A_i=\bigcup_{a\in A}A_i^a=\bigcup_{c\in A}C_i^c\subseteq C_i$, whence we conclude that $C_i\in [A_i,A_i\cup B_i]\subseteq\DCH(A_i,B_i)$.
\end{proof}

On the other hand, it is not the case that all local Kripke-operators weakly preserve supported convexity. This is seen in the next example.

\begin{example}
\begin{enumerate}[(a)]
\item 
Let $X=\{a,b\}$, and let $\cR$ be the relation
$\{(Y,X)\mid Y\not=\emptyset\}\subseteq \cP(X)^2$. Then $\Delta_\cR$
is clearly local, but it does not weakly preserve supported convexity,
since the family $\SCH(X)=\{X\}$ is convex and supported, but its
image $\Delta_\cR(\{X\})=\{\{a\},\{b\},X\}$ is not supported.

\item
More generally, if $\Delta_\cR\colon\cP(\cP(X))^n\to\cP(\cP(Y))$ is
local and there are tuples \hfill\break
$(\{a\},A_0,\ldots,A_{n-1}),(\{b\},B_0,\ldots,B_{n-1})$ in $\cR$ such
that $a\not=b$ and $A_i\cap B_i\not=\emptyset$ for some $i<n$, then
$\Delta_\cR$ does not weakly preserve supported convexity. This is
because by Definition \ref{local},
$\cR[\emptyset]=\{(\emptyset,\ldots,\emptyset)\}$, whence
$\emptyset\notin\Delta_\cR(\SCH(A_0,B_0),\ldots,\SCH(A_{n-1},B_{n-1}))$
even though $\emptyset\in\SCH(\{a\},\{b\})$.
\end{enumerate}
\end{example}

To avoid the problem exhibited in the example above, we consider the following additional requirement for (local) Kripke-operators:

\begin{definition}
A Kripke-operator $\Delta_\cR\colon \cP(\cP(X))^n\to\cP(\cP(Y))$ is \emph{separating} if $A_i\cap B_i=\emptyset$
for all $i<n$ whenever $(A_0,\ldots,A_{n-1})\in\cR[\{a\}]$, $(B_0,\ldots,B_{n-1})\in\cR[\{b\}]$ and $a\not=b$.
\end{definition}

\begin{theorem}\label{localseparating}
If $\Delta_\cR\colon \cP(\cP(X))^n\to\cP(\cP(Y))$ is a local and separating Kripke-operator for finite $X$ and $Y$, then it weakly preserves
supported convexity.
\end{theorem}

\begin{proof}
We show that $\cR$ satisfies the condition
$(*^\flat)$ of Lemma \ref{supcon}. Thus, assume that $(A_0,\ldots,A_{n-1})\in \cR[A]$, $(B_0,\ldots,B_{n-1})\in \cR[B]$, and $C\in\SCH(A,B)=[A\cap B,A]\cup[A\cap B,B]$. We consider the case
$C\in [A\cap B,A]$; the other case is similar. 

By Definition \ref{local}, for each $i<n$ and each $a\in A$ there are sets
$A_i^a$ such that $(A_0^a,\ldots,A_{n-1}^a)\in \cR[\{a\}]$ and $A_i=\bigcup_{a\in A}A^a_i$. Similarly, for each $b\in B$, there are sets $B_i^b$ such that $(B_0^a,\ldots,B_{n-1}^a)\in \cR[\{b\}]$ and $B_i=\bigcup_{b\in B}B^b_i$. We define now $C_i=\bigcup_{c\in C} A^c_i$ for $i<n$. Then by Definition \ref{local} we have $(C_0,\ldots,C_{n-1})\in \cR[C]$.

It is clear from the definition that $C_i\subseteq A_i$. Thus, to complete the proof it suffices to show that $A_i\cap B_i\subseteq C_i$ for each $i<n$. To show this, assume that $d\in A_i\cap B_i$. Then there are elements $a\in A$ and $b\in B$ such that $d\in A_i^a$ and $d\in B_i^b$. As $\Delta_\cR$ is separating this implies that $a=b\in A\cap B\subseteq C$, whence $d\in A_i^a=C_i^a\subseteq C_i$.
\end{proof}

Recall from Example \ref{kripke-ex}
the Kripke-relations $\cR_\cap$,  $\cR_\lor$  and $\cR_{\cK,\vec\ell}$ that define the Kripke-operators that correspond to conjunction, (tensor) disjunction and quantification with the Lindström quantifier $Q_\cK$. 
We prove next that the corresponding Kripke operators are local and separating.

\begin{proposition}\label{LocSepProp}
The operators $\Delta^{M^m}_\cap$, $\Delta^{M^m}_\lor$ and $\Delta^{M^m}_{\cK,\vec\ell}$ are local and separating.
\end{proposition}

\begin{proof}
Note first that $\cR_\cap[A]=\{(A,A)\}$ for any $A\in \cP(M^m)$. Hence
we have $(A_0,A_1)\in\cR_\cap[A]$ if and only if
$A_0=A_1=A=\bigcup_{\va\in A}\{\va\}$ if and only if for each $\va\in A$
there are sets $A^{\va}_0,A^{\va}_1$ such that
$(A^{\va}_0,A^{\va}_1)\in\cR_\cap[\{\va\}]$ and $A_i=\bigcup_{\va\in A}A^{\va}_i$ for
$i<2$. Thus $\Delta^{M^m}_\cap$ is local. Since
$\cR_\cap[\{\va\}]=\{(\{\va\},\{\va\})\}$, it is clearly separating, too.

Consider then the tensor disjunction operator on $M^m$. By the definition of $\cR_\lor$ we have
$\cR_\lor[\{\va\}]=\{(\{\va\},\{\va\}),(\{\va\},\emptyset),(\emptyset,\{\va\})\}$
for any $\va\in M^m$.  Using this it is straightforward to verify that
$\Delta^{M^m}_\lor$ is local and separating.

Finally we show that the $(\cK,\vec\ell)$-projection operator $\Delta^{M^m}_{\cK,\vec\ell}$ is separating and local.
Assume first that $(M,\emptyset)\notin\cK$. Then by the definition of $\cR_{\cK,\vec\ell}$ we see that for any tuple $\vec a\in M^{m-r}$,
$\cR_{\cK,\vec\ell}\,[\{\vec a\}]=\{S\in\pot(M^m)\mid \pi^p_{\cK,\vec\ell}\,(S)=\{\vec a\}\}$. Clearly $S\cap S'=\emptyset$ if $\pi^p_{\cK,\vec\ell}\,(S)=\{\vec a\}$ and $\pi^p_{\cK,\vec\ell}\,(S')=\{\vec b\}$ for $\vec a\not=\vec b$, whence $\Delta^{M^m}_{\cK,\vec\ell}$ is separating.

To show locality, observe that $S\in \cR_{\cK,\vec\ell}\,[T]$ if and only if 
$T=\pi^p_{\cK,\vec\ell}\,(S)=\{\vec a\in M^{m-r}\mid S[\vec a]_{\vec\ell}\in \cK\}$. Assume first that this equality
holds. Then $S=\bigcup_{\vec a\in T}S^{\vec a}$, where $S^{\vec a}:=
\{\vc\in S\mid\vc=\va\otimes_{\vec\ell}\,\vb \text{ for some }\vb\in M^r\}$. Moreover, the equality implies that $S[\vec a]_{\vec\ell}\in \cK$, whence 
$S^{\vec a}\in \cR_{\cK,\vec\ell }\,[\{\vec a\}]$ for all $\vec a\in A$. 

Assume then that $S=\bigcup_{\vec a\in T}S^{\vec a}$ for some sets 
$S^{\vec a}\in \cR_{\cK,\vec\ell}\,[\{\vec a\}]$, $\vec a\in A$. Then by definition $\pi^p_{\cK,\vec\ell}\,(S^{\va})=\{\va\}$ for each $\vec a\in A$, whence $\pi^p_{\cK,\vec\ell}\,(S)=\bigcup_{\va\in T}\pi^p_{\cK,\vec\ell}\,(S^{\va})=T$,
and consequently $S\in\cR_{\cK,\vec\ell}\,[T]$.

In the case $(M,\emptyset)\in\cK$, we have $\cR_{\cK,\vec\ell}\,[T]=\{S\in\pot(M^m)\mid \pi^p_{\cK,\vec\ell}\,(S)\subseteq T\}$. This just means that  $\emptyset$ is added to $\cR_{\cK,\vec\ell}\,[\{\vec a\}]$ for each $\va\in M^{m-r}$.
Clearly this does not affect the proof that $\Delta^{M^m}_{\cK,\vec\ell}$ is separating. The proof of locality also goes through by defining $S^{\va}=\emptyset$ for $\va\in T\jer \pi^p_{\cK,\vec\ell}\,(S)$.
\end{proof}

We end this section by showing that not all of the Kripke-operators of Example~\ref{kripke-ex} are local and separating.

\begin{example}\label{nonlocal}

\begin{enumerate}[(a)]

\item
Consider the restricted union operator  $\Delta_{\cR_{\cup^*}}$ on a base set $X$. By the
definition of $\cR_{\cup^*}$ we have
$\cR_{\cup^*}[\{a\}]=\{(\{a\},\emptyset),(\emptyset,\{a\})\}$ for
any $a\in X$. Hence $\Delta_{\cR_{\cup^*}}$ is clearly
separating. However, it is \emph{not} local: if $a\not=b$, then
$(\{a\},\emptyset)\in\cR_{\cup^*}[\{a\}]$ and
$(\emptyset,\{b\})\in\cR_{\cup^*}[\{b\}]$, but
$(\{a\},\{b\})=(\{a\}\cup\emptyset,\emptyset\cup\{b\})
\notin\cR_{\cup^*}[\{a,b\}]$.

\item
Tensor conjunction on a base set $X$ with at least three elements 
is not local: if $a,b,c\in X$ are distinct elements, then 
$(\{a,b\},\{a,c\})\in \cR_{\land}[\{a\}]$
and $(\{b,c\},\{a,b\})\in \cR_{\land}[\{b\}]$, but
$(\{a,b\}\cup\{b,c\},\{a,c\}\cup\{b,c\})=(\{a,b,c\},\{a,b,c\})\notin \cR_{\land}[\{a,b\}]$. It is neither
separating as the intersection of first components
$\{a,b\}$ and $\{b,c\}$ (as well as that of the second components) is nonempty.

By a similar argument we see that tensor negation and other non-monotone tensor operators are neither local nor separating.

Note however that, as mentioned in 
Example~\ref{tensorintpres}, all tensor operators weakly preserve intervals. Moreover, we will later prove that tensor conjunction weakly preserves both dominated and supported convexity (see Proposition~\ref{cap-lor-land}).

\end{enumerate}
\end{example}

\subsection{Logical operators preserving dimensions}

We are now ready to prove that the basic logical operators of first-order logic (except for negation), as well as arbitrary Lindström quantifiers, preserve growth classes.

\begin{corollary}\label{DimPresCor}
Let $\OO$ be a growth class, and let $\gdim$ be one of the dimension functions $\Dim$, $\Dimd$ and $\CDim$. Furtheoremore, let $\phi=\phi(\vx)$ and $\psi=\psi(\vx)$ be formulas of some logic $\cL$ with team semantics.
\begin{enumerate}[(a)]
    \item If $\gdim_{\phi,\vx},\gdim_{\psi,\vx}\in\OO$, then $\gdim_{\phi\land\psi,\vx}\in\OO$.
    \item If $\gdim_{\phi,\vx},\gdim_{\psi,\vx}\in\OO$, then  $\gdim_{\phi\lor\psi,\vx}\in\OO$.
    \item If $\gdim_{\phi,\vx}\in\OO$, then $\gdim_{\exists x_i\phi,\vx^-}\in\OO$ and $\gdim_{\forall x_i\phi,\vx^-}\in\OO$, where $\vx^-$ is $\vx$ without the component $x_i$.
    \item If $Q_\cK$ is a Lindström quantifier, $\vx=\vz\otimes_{\vec\ell}\vy$ and $\gdim_{\phi,\vx}\in\OO$, then $\gdim_{Q_\cK\vy\,\phi,\vz}\in\OO$.
\end{enumerate}
\end{corollary}

\begin{proof}
(a) Let $M$ be a finite model, and let $\len(\vx)=m$. By Proposition \ref{LocSepProp}, the operator $\Delta^{M^m}_\cap$ is local and separating, whence by Theorems \ref{LocalDomConPres} and \ref{localseparating} it weakly preserves both dominated and supported convexity. Thus it follows from Theorem \ref{sailyminen} that 
$$
    D(\sat{\phi\land\psi}{M,\vx})=D(\sat{\phi}{M,\vx}\cap\sat{\psi}{M,\vx})\le D(\sat{\phi}{M,\vx})\cdot D(\sat{\psi}{M,\vx})
$$
for each of the dimensions $D\in\{\DD,\DDd,\CD\}$.
Since this holds for all finite models $M$, we have $\gdim_{\phi\land\psi,\vx}\le \gdim_{\phi,\vx}\cdot\gdim_{\psi,\vx}$, and hence $\gdim_{\phi\land\psi,\vx}\in\OO$. 

(b) is proved in the same way as (a).

(c) follows from (d) as a special case.

(d) Let $M$ be a finite model. As in (a), it follows from Proposition~\ref{LocSepProp} and Theorems~\ref{LocalDomConPres} and \ref{localseparating} that the operator $\Delta^{M^m}_{\cK,\vec\ell}$ weakly preserves both dominated and supported convexity, whence using Theorem~\ref{sailyminen}, we see that
$$
    D(\sat{Q_\cK\vy\,\phi}{M,\vz})=D(\Delta^{M^m}_{\cK,\vec\ell}(\sat{\phi}{M,\vx}))\le D(\sat{\phi}{M,\vx})
$$
for each of the dimensions $D\in\{\DD,\DDd,\CD\}$. Hence $\gdim_{Q_\cK\vy\,\phi,\vz}\le\gdim_{\phi,\vx}$, and consequently $\gdim_{Q_\cK\vy\,\phi,\vz}\in\OO$.
\end{proof}

The list of logical operators that preserve growth classes of dimensions can be extended by simply appealing to basic definitions. We have already seen in Example~\ref{tensorintpres} that all tensor operators (weakly) preserve intervals. Moreover, in spite of the fact that tensor conjunction is not local (see Example~\ref{nonlocal}(b)), we can prove that it weakly preserves both dominated and supported convexity. 

\begin{proposition}\label{cap-lor-land}
The operator $\Delta_\land$ weakly preserves dominated convexity and supported convexity.
\end{proposition}

\begin{proof}
Assume that $\cA_0$ and $\cA_1$ are dominated and convex. We show first that $\Delta_\land(\cA_0,\cA_1)$ is convex. Thus, assume that $A\subseteq C\subseteq B$ and $A,B\in\Delta_\land(\cA_0,\cA_1)$. Then there are $A_0,B_0\in\cA_0$ and $A_1,B_1\in\cA_1$ such that $A=A_0\cap A_1$ and $B=B_0\cap B_1$. Let $C_0=A_0\cup C$ and
$C_1=A_1\cup C$.  Then 
$A_0\subseteq C_0$ and $C_0\subseteq A_0\cup B\subseteq A_0\cup B_0$, and since $\cA_0$ is dominated and convex, $A_0\cup B_0\in\cA_0$. Thus, by convexity of $\cA_0$, we have $C_0\in\cA_0$. In the same way we see that $C_1\in\cA_1$. Observe now that $C\subseteq C_0\cap C_1\subseteq (A_0\cap A_1)\cup C= A\cup C=C$, whence $C\in\Delta_\land(\cA_0,\cA_1)$.

To prove that $\Delta_\land(\cA_0,\cA_1)$ is dominated, it suffices to observe that if $\cA_0$ and $\cA_1$ are dominated by $D_0$ and $D_1$, respectively, then clearly $\Delta_\land(\cA_0,\cA_1)$ is dominated by $D_0\cap D_1$.

The proof that $\Delta_\land$ weakly preserves supported convexity is similar.
\end{proof}

Finally, for the union operator we obtain the following dimension inequalities:

\begin{proposition}\label{uniondim}
Let $\cA_0,\cA_1\subseteq \cP(X)$ for a base set $X$, and let $\cA=\cA_0\cup\cA_1$. Then $\DD(\cA)\le\DD(\cA_0)+\DD(\cA_1)$, $\DDd(\cA)\le\DDd(\cA_0)+\DDd(\cA_1)$ and $\CD(\cA)\le\CD(\cA_0)+\CD(\cA_1)$.
\end{proposition}

\begin{proof}
Observe that if a subfamily $\cG_0$ dominates $\cA_0$ and a subfamily $\cG_1$ dominates $\cA_1$,
then clearly $\cG_0\cup\cG_1$ dominates $\cA_0\cup\cA_1$. Thus, $\DD(\cA)\le |\cG_0\cup\cG_1|\le |\cG_0|+|\cG_1|$. The first inequality follows from the case where $\cG_0$ and $\cG_1$ are of minimal cardinality. The other two inequalities are proved in the same way.
\end{proof}

We can now add the cases of tensor connectives and intuitionistic disjunction to Corollary~\ref{DimPresCor}.

\begin{corollary}\label{DimPresCor2}
Let $\OO$ be a growth class, and let $\gdim$ be one of the dimension functions $\Dim$, $\Dimd$ and $\CDim$. Furtheoremore, let $\phi=\phi(\vx)$ and $\psi=\psi(\vx)$ be formulas of some logic $\cL$ with team semantics, and let $\oast$ be a binary tensor connective.
\begin{enumerate}[(a)]
    \item If $\gdim_{\phi,\vx},\gdim_{\psi,\vx}\in\OO$, then $\gdim_{\phi\ilor\psi,\vx}\in\OO$.
    \item If $\gdim_{\phi,\vx},\gdim_{\psi,\vx}\in\OO$, then  $\gdim_{\phi\tland\psi,\vx}\in\OO$.
    \item If $\CDim_{\phi,\vx},\CDim_{\psi,\vx}\in\OO$, then  $\CDim_{\phi\oast\psi,\vx}\in\OO$.
\end{enumerate}
\end{corollary}

\begin{proof}
(a) Let $M$ be a finite model. By Proposition \ref{uniondim}  we have
$$
    D(\sat{\phi\ilor\psi}{M,\vx})=D(\sat{\phi}{M,\vx}\cup\sat{\psi}{M,\vx})\le D(\sat{\phi}{M,\vx})+ D(\sat{\psi}{M,\vx})
$$
for each of the dimensions $D\in\{\DD,\DDd,\CD\}$.
Since this holds for all finite models $M$, we have $\gdim_{\phi\ilor\psi,\vx}\le \gdim_{\phi,\vx}+\gdim_{\psi,\vx}$, and hence $\gdim_{\phi\ilor\psi,\vx}\in\OO$. 

(b) Using Proposition \ref{cap-lor-land} and Theorem \ref{sailyminen} we obtain the inequality $\gdim_{\phi\tland\psi,\vx}\le \gdim_{\phi,\vx}\cdot\gdim_{\psi,\vx}$. Thus we see that $\gdim_{\phi\tland\psi,\vx}\in\OO$.

(c) is proved in the same way as (b) by using Proposition~\ref{basic:tensor-intpres} (see Example \ref{tensorintpres}) in place of Proposition~\ref{cap-lor-land}. 
\end{proof}




\section{Applications}

The main application of our dimension theory is to hierarchies of definability in logics based on the atoms of Definition~\ref{atoms} and the logical operations of Definition~\ref{fol}. We obtain also non-expressibility results for some other connectives and quantifiers based on observations that they do not preserve dimension.

\subsection{Hierarchy results}

We can now apply our results to obtain hierarchy results for
extensions of first order logic by various team-based atoms. We start
by defining a family of logics the definition of which is based solely
on dimension-theoretic considerations. We use these somewhat
artificial logics as yardsticks to compare more traditional logics.




\begin{definition}\label{massive}
\begin{enumerate}[(a)]
    \item The logic $\LE^U_k$ is the closure of literals  and all atoms whose upper dimension function is in the growth class $\EE_k$ under the connectives $\wedge$, $\ilor$, $\vee$, $\tland$, and any Lindstr\"om quantifiers. Similarly $\LF^U_k$ for $\FF_k$.
       \item The logic $\LF^D_k$ is the closure of literals and all atoms whose dual dimension function is in the growth class $\FF_k$ under the connectives $\wedge$, $\ilor$, $\vee$, $\tland$, and any Lindstr\"om quantifiers. 
            \item The logic $\LF^C_k$ is the closure of literals and all atoms whose cylindrical dimension function is in the growth class $\FF_k$ under the connectives $\wedge$, $\ilor$, $\vee$, any tensor operators, and any Lindstr\"om quantifiers. 
           
\end{enumerate}

\end{definition}

We did not define what would be denoted $\LE^D_k$ and $\LE^{C}_k$, for the very special reason that the estimates given by Proposition~\ref{calc:dummydf}  are not good enough for the dual and the cylindric dimensions, rendering logics based on them less natural. See remarks at the end of Subsection~\ref{gc}.  

The logics defined above have some unusual properties. For example, each logic is closed under \emph{all} Lindstr\"om quantifiers which means that \emph{every} property of finite models, closed under isomorphism, is definable in each of these logics. On the other hand, each of these logics is limited as to what their formulas can express. In classical logic formulas and sentences have more or less the same expressive power because we can always form a sentence from a formula by substituting constant symbols in place of free variables. In team semantics this does not work because constant symbols do not convey the plural nature of team semantics. The  reason for the introduction of these logics is that they help us estimate and delineate dimensions of formulas and thereby expressive power of formulas in a multitude of logics.

\begin{theorem}
\begin{enumerate}[(a)]
    \item The upper  dimension of every formula in $\LE^U_k$ 
     is in the growth class~$\EE_k$.
    \item The upper (dual, cylindrical) dimension of every formula in $\LF^U_k$ 
    ($\LF^D_k$, $\LF^{CD}_k$, respectively) is in the growth class $\FF_k$.\end{enumerate}
\end{theorem}

\begin{proof}
(a) By Definition~\ref{massive} the atoms of $\LE^U_k$ 
     are in  $\EE_k$. By an inductive argument based on Corollaries~\ref{DimPresCor} and \ref{DimPresCor2} the upper dimension of every formula from $\LE^U_k$ is in $\EE_k$, too.
     
(b) By Definition~\ref{massive} the atoms of $\LF^U_k$ 
     are in  $\FF_k$. By an inductive argument based on Corollaries~\ref{DimPresCor} and \ref{DimPresCor2}, again, the upper dimension of every formula of $\LF^U_k$ is in  $\FF_k$, too. The argument is the same in the case of $\LF^D_k$ and $\LF^{CD}_k$.
\end{proof}

Note that we have not added the intuitionistic implication $\to$ (see Definition~\ref{IntImpl}) to the lists of logical operations in the above definition. The reason is that we want to keep dimension under control and  intuitionistic implication increases dimension exponentially (Lemma~\ref{exponential}). The non-empty atom $\nem$ is in  $\LE^U_0$.
For $k>0$ the logics $\LE^U_k$, $\LE^D_k$, and $\LE^C_k$, $\LF^U_k$, $\LF^D_k$, and $\LF^C_k$ are closed under $\exists^1$, but never under $\forall^1$ (see Section~\ref{otheratomsand}.) 

The trivial properties of the logics of Definition~\ref{massive} are summarized in the following lemma (see also Figure~\ref{hasse}):

\begin{lemma}
\begin{enumerate}[(a)]
    \item $\LF^U_k\subseteq\LF^U_{k+1}$, $\LF^D_k\subseteq\LF^D_{k+1}$, and $\LF^C_k\subseteq\LF^C_{k+1}$.
    \item $\LE^U_k\subseteq\LF^U_k\subseteq \LE^U_{k+1}$.    
    \item $\LF^C_k\subseteq\LF^U_k$ and $\LF^C_k\subseteq\LF^D_k$.
    \end{enumerate}
\end{lemma}

\begin{figure}
    \centering
\begin{tikzpicture}
  \matrix (m) [matrix of math nodes, row sep=3em,
    column sep=3em]{
                        \LF^U_{k+1}&                  &   \LF^D_{k+1}  \\     
   \LE^U_{k+1}                     &     \LF^C_{k+1}       &              \\
                        \LF^U_{k}  &                  &   \LF^D_{k}  \\
      \LE^U_{k}                    &    \LF^C_{k}     &               \\
};
  \path[-stealth]
(m-2-1) edge (m-1-1)
(m-3-1) edge (m-2-1)
(m-4-1) edge (m-3-1)
(m-4-2) edge (m-3-1) edge (m-3-3) edge (m-2-2)
(m-3-3) edge (m-1-3)
(m-2-2) edge (m-1-1) edge (m-1-3)

 ; \end{tikzpicture}
\vspace{10mm}
\caption{Logics built from growth classes.\label{hasse}} \end{figure}
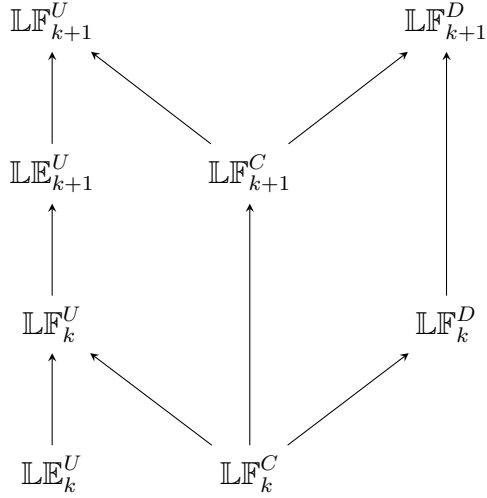

As it turns out, a crucial factor in the hierarchy results is the length of variable-tuples allowed in the atoms. Let us therefore specify the concept of arity for our atoms: 

\begin{definition}
We say:
\begin{itemize}
    \item 
 the atom $\dep(\vx,y)$ is $k$-ary, if $\len(\vx)=k$, 
 \item the atoms $\vx\ |\ \vy$ and $\vx\ \Upsilon\ y$ are $k$-ary if $\len(\vx)(=\len(\vy))=k$, 
 \item the atom $ \vt_2\perp_{\vt_1}\vt_3$ is $m+\max(k,l)$-ary, or alternatively $(k,l,m)$-ary,  if $\len(\vt_1)=m,\len(\vt_2)=k,\mbox{ and }\len(\vt_3)=l$,
 \item the atom 
$ \vt_2\perp\vt_3$ is $\max(k,l)$-ary, or alternatively $(k,l)$-ary,  if $\len(\vt_2)=k,\mbox{ and }\len(\vt_3)=l$, 
\item a general atom $\alpha_C\vec{x}$ (as in Definition~\ref{atoms}) $\len(\vec{x})=k$, is called $k$-ary,
\item a logic is $k$-ary (respectively, $(k,l)$-ary or $(k,l,m)$-ary) if its atoms are.
\end{itemize}\end{definition}


\begin{theorem} 
\begin{enumerate}[(a)]
    \item $k$-ary inclusion, anonymity, and exclusion  logics are all included in $\LE^U_k$.
       \item The $k$-ary dependence logic is included in $\LF^U_k$.
\item The $(k,l,m)$-ary independence logic is included in $\LF^U_{\max(k,l)+m}$.
\end{enumerate}
\end{theorem}

\begin{proof}
(a)  By Theorem~\ref{dim-atoms} the atoms of the  $k$-ary inclusion, anonymity, exclusion and  independence logics have upper dimension in  $\EE_k$. 
     
(b) The proof here is entirely similar: By Theorem~\ref{dim-atoms}   the  $k$-ary dependence atom has upper dimension in  $\FF_k$. 
     \end{proof}

The following theorem is our main application of the dimension analysis of families of sets of $n$-tuples.

\begin{theorem}
\begin{enumerate}[(a)]
    \item 
The $k+1$-ary inclusion, anonymity, and exclusion atoms are not definable in $\LE^U_k$.
    \item The $k+1$-ary dependence atom is not definable in $\LF^U_k$.
    \item The $(k,l,m)$-ary independence atom is not definable in $\LF^U_{i}$ if $i<\max(k,l)+m$.
\end{enumerate}
\end{theorem}

\begin{proof}
Suppose $\len(\vec{x})=\len(\vec{y})=k+1$. By Theorem~\ref{dim-atoms} the upper dimension of $\sat{\vec{x}\subseteq\vec{y}}{\vec{x}\vec{y}}$ is $2^{n^{k+1}}-n^{k+1}$. Therefore $\Dim_{\vec{x}\subseteq\vec{y},\vec{z}}\notin \EE_k$. The argument is the same in the other cases.
\end{proof}

Despite the above non-definability results, there are some obvious and also some not so obvious inter-definability results between the atoms. The basic picture is that dependence atoms are definable from the independence atoms but not from the inclusion atoms. The inclusion atoms are definable from the independence atoms but not from the dependence atoms. In both cases the non-definability is a consequence of structural properties of the logics, namely, dependence logic is downward closed and inclusion logic is closed under unions (of teams). The known relationships are as follows:

\begin{proposition}[ \cite{Ga}]\label{translations}
\begin{enumerate}[(a)]
\item The $k$-ary dependence atom $\dep(\vx,y)$ is definable from the $k+1$-ary
  exclusion atom with the formula
  $$\forall z(z=y\vee \vx z\ |\ \vx y)$$ and also in terms of the $k+1$-ary pure independence
  atom with the formula\footnote{Here, as in the sequel, $\vz=\vx$ is shorthand for $\bigwedge_{i=1}^k z_i=x_i$ and, respectively, $\vz\ne\vx$ is shorthand for $\bigvee_{i=1}^k \neg z_i=x_i$ } $$  \forall\vz\ \exists w((\vz\ne\vx\vee w=y)\wedge\vz y\perp\vz w).$$ In the other direction, the $k$-ary exclusion atom $\vt_1\ |\ \vt_2$ is
  definable from the $k$-ary dependence atom with the formula $$\forall\vz\exists u_1 u_2(\dep(\vz,u_1)\wedge \dep(\vz,u_2)\wedge((u_1=u_2\wedge\vz\ne \vt_1)\vee(
u_1\ne u_2\wedge\vz\ne\vt_2))).$$
\item The $k$-ary exclusion atom $\vx\ |\ \vy$ can be defined in terms of the
  $k$-ary inclusion and the $k$-ary pure independence atoms with the formula:
  $$\exists \vz(\vx\subseteq \vz\wedge \vy\perp\vz\wedge\vy\ne\vz).$$
\item The $k$-ary inclusion atom  $\vt_1\subseteq\vt_2$ can be defined from the ($k$,2)-ary  pure independence atom with the formula
$$\forall v_1v_2\vz((\vz\ne\vt_1\wedge \vz\ne\vt_2)\vee(v_1\ne v_2\wedge
\vz\ne \vt_2)\vee ((v_1=v_2\vee\vz=\vt_2)\wedge \vz\perp v_1v_2)).$$ It is also definable from the $k$-ary anonymity atom with the formula (\cite{raine})
$$\exists x\forall y(x=y)\vee\forall w_1\forall w_2\exists \vy\exists z(((w_1=w_2\wedge\vy=\vt_1)\vee(\neg w_1=w_2\wedge\vy=\vt_2))\wedge\vy\anonym {z})).$$

\item The $k$-ary anonymity atom $\vx\anonym{y}$ is definable in terms of the
  $k+1$-ary inclusion atom with the formula
  $$\exists u(\neg u=y\wedge \vx u\subseteq \vx y).$$
\item The $(k,l,m)$-ary independence atom $\vt_2\perp_{\vt_1}\vt_3$ is definable in terms of the
  $k+l+m$-ary dependence atom,  $k+m$-ary exclusion
  atoms, and the $k+l+m$-ary inclusion atom with the formula
 $$\begin{array}{l}
 \forall\vp\vq\vr\ \exists u_1 u_2 u_3 u_4((\bigwedge_{i=1}^4\dep(\vp\vq\vr,u_i))
\wedge((u_1\ne u_2\wedge(\vp\vq\ |\ \vt_1\vt_2))\vee\\
(u_1=u_2\wedge u_3\ne u_4\wedge
(\vp\vr \ |\ \vt_1\vt_3))\vee\\
(u_1=u_2\wedge u_3=u_4\wedge
(\vp\vq\vr\subseteq\vt_1\vt_2\vt_3)))). \end{array}$$

\item The $(k,l,m)$-ary  independence atom
  $\vx\perp_{\vz}\vy$ is definable in terms of the  $(k+m,l+m)$-ary pure independence atom with the formula
    (\cite{Wi})
    $$\forall\vp\vq\exists\vu\exists\vw((\vz\ne\vp\vee\vz\ne\vq\vee\vu\vw=\vx\vy)\wedge(\vz\ne\vp\vee\vz\ne\vq\vee\vp\ne\vq\vee\vz=\vp)\wedge\vp\vu
    \perp\vq\vw)).$$
\end{enumerate}
\end{proposition}

Note that (a) above is in harmony
with Theorem~\ref{dim-calculations}, as for $n>2$
$$2^{n^k}-2< n^{n^k}< 2^{n^{k+1}}-2.$$

\begin{corollary}
[Hierarchy Theorem]
Dependence logic, exclusion logic, inclusion logic, anonymity logic and  pure independence logic each has a proper definability hierarchy for formulas  based on the arity of the non-first order atoms.
\end{corollary}

The Corollary holds in fact in a stronger form:

\begin{theorem}\label{hierarchy}
Suppose $k$ is a positive integer.
\begin{enumerate}[(a)]
 \item The $k$-ary dependence atom is not definable in the extension of first order logic by $<k$-ary dependence (or any other\footnote{See Definition~\ref{atoms}.} $<k$-ary) atoms,  $\le k$-ary independence, exclusion, inclusion, anonymity, constancy atoms, and any Lindstr\"om quantifiers.
 \item The $k$-ary exclusion atom is not definable in the extension of first order logic by $<k$-ary exclusion,  inclusion, anonymity, dependence, independence, constancy (or any other $<k$-ary) atoms, and any Lindstr\"om quantifiers.
  \item The $k$-ary inclusion atom is not definable in the extension of first order logic by $<k$-ary inclusion, exclusion, anonymity, dependence, or constancy (or any other $<k$-ary)  atoms, and any Lindstr\"om quantifiers.
  \item The $k$-ary anonymity atom is not definable in the extension of first order logic by $<k$-ary inclusion, anonymity, exclusion, dependence, constancy (or any other $<k$-ary)  atoms, and any Lindstr\"om quantifiers.

 \item The $k$-ary  independence atom (whether pure or not) is not definable in the extension of first order logic by $<k$-ary independence, inclusion, anonymity, exclusion, dependence, constancy (or any other $<k$-ary) atoms, and any Lindstr\"om quantifiers.

\end{enumerate}
\end{theorem}


There are many open problems arising from comparing the definability results of Lemma~\ref{translations} and the non-definability results of Theorem~\ref{hierarchy}. We mention a few in Section 9.
  
Theorem~\ref{hierarchy} shows that the translations in Lemma~\ref{translations} necessarily involve increase of arity.

Earlier hierarchy results have been mostly for sentences. In \cite{DK} it is shown  that $k$-ary dependence atom is weaker
than $k+1$-ary dependence atom for {\bf sentences} in vocabulary
having arity $k+1$.  In \cite{GHK} it is shown
 that independence logic with $k$-ary independence atoms is
strictly weaker than independence logic with $k+1$-ary independence
atoms on the level of sentences.
In \cite{Ha}
it is shown (using similar results from \cite{Gr} on transitive closure and
fixpoint operator) that inclusion logic with $k-1$-ary inclusion atoms
is strictly weaker than inclusion logic with $k$-ary inclusion atoms
for sentences when $k\ge2$.  As to earlier results for formulas, in \cite[Theorem 5.17, Corollary 5.18]{raine} it is shown that the fullness (the property of containing \emph{every} assignment of the appropriate kind) of a team with domain $\{x_1,\ldots,x_{k+1}\}$, which can be defined by means of the $k+1$-ary inclusion atom, cannot be defined in the extension of first order logic by what are called $k$-invariant atoms in \cite{raine} and any downward closed atoms.

\subsection{Other atoms and logical operations}\label{otheratomsand}

The atoms and logical operations  $\wedge$, $\ilor$, $\vee$, $\forall$, and $\exists$ are by no means the only ones that can be or have been considered. In this section we first introduce two new atoms that have particularly big upper or other dimension. We then show that many other logical operations occurring in the literature actually fail to preserve dimension. We use this to conclude some interesting non-definability results concerning these alternative logical operations.


\subsubsection*{Intuitionistic implication and disjunction}
\smallskip

\begin{definition}[Intuitionistic implication]\label{IntImpl}The intuitionistic implication $\phi\to\psi$ is defined by
$M\models_T\phi\to\psi$ if and only if every $Y\subseteq T$ that satisfies in $M$ the formula $\phi$ satisfies also the formula $\psi$.

\end{definition}

As the following lemma demonstrates, the dependence atom can be defined in terms of the constancy atoms and the intuitionistic implication:

\begin{lemma}[\cite{AV}]\label{exponential}
$\models \dep(x_1,\ldots,x_n,y)\leftrightarrow\footnote{We use $\models\phi\leftrightarrow\psi$ as a shorthand to ``For all models $M$ and all teams $T$, $M\models_T\phi$ if and only if $M\models_T\psi$".} \bigl((\dep(x_1)\wedge\ldots\wedge\dep(x_n))\ \to\ \dep(y)\bigr)$
\end{lemma}

This gives an example where the use of $\phi\to\psi$ leads to something we know is exponential (Example~\ref{atomiendimensioita}). It shows that we cannot hope to prove 
that the dimension of $\phi\to\psi$ is in general better than exponential in the dimensions of $\phi$ and $\psi$. 

Note, that we can add intuitionistic implication to $\LE_0$, because it does not increase upper dimension, when the latter is bounded by a constant.


Intuitionistic disjunction can be defined in terms of constancy atoms:
$$\models\phi\ \ilor\ \psi\leftrightarrow\exists x\exists y(\dep(x)\wedge\dep(y)\wedge ((x=y\wedge \phi)\vee(\neg x=y\wedge\psi))).$$
But since it increases upper dimension additively, it cannot be defined in first order logic alone. In fact, the formula
$x=y\ \ilor\ \neg x=y$ has upper dimension 2. 

\subsubsection*{The non-empty atom $\nem$}
\smallskip

\begin{definition}[The non-empty atom]\label{non-empty-atom}
The non-empty atom $\nem$ is defined by $M\models_T\nem$ if and only if $T\ne\emptyset$.
\end{definition}

The atom $\nem$ says that a team is non-empty. Most of
the atoms we have considered (dependence, inclusion, independence,
etc) satisfy the Empty Team Property, i.e., the empty team satisfies
the atom (see the remark in the end of Section~\ref{families-of-teams}) and our logical operations (conjunction, disjunction,
existential quantifier, universal quantifier) preserve the Empty Team
Property. Thus we can immediately observe that $\nem$ is not
definable in them. Still it is sometimes useful. For example, we may
want to enhance the disjunction $\phi\vee\psi$ to
$(\phi\wedge\nem)\vee(\psi\wedge\nem)$. The latter would be
satisfied by a team which splits into a team satisfying $\phi$ and a
team satisfying $\psi$, both non-empty. An example in natural language
would be the statement ``On Mondays I play tennis or go to swim" with
the intention that both cases actually happen.

\begin{lemma}
The upper dimension of $\nem$ is $1$. The dual upper dimension $\DDd(\sat{\nem}{{M,\vx}})$ and the cylindric dimension $\CD(\sat{\nem}{{M,\vx}})$ in a domain of size $n$ are $n^k$, where $k=\len(\vx)$.
\end{lemma}

\begin{proof}
Non-emptyness is a convex property dominated by the maximal team. Hence the upper dimension of $\nem$ is 1. It is  supported by the family of all singleton teams. Hence the dual upper dimension and the cylindrical dimension of $\sat{\nem}{{M,\vx}}$, $\len(\vx)=k$, is $n^k$.
\end{proof}

\begin{corollary} 
$\Dim_{\nem,{\vx}}$ is in $\mathbb{E}_0$ while
$\Dimd_{\nem,{\vx}}$ and $\CDim_{\nem,{\vx}}$ are in
 $\mathbb{F}_0$.   
\end{corollary}


The atom $\nem$ is an example of an upper dimension 1 operation which still extends the expressive power of first order logic.


\subsubsection*{The quantifiers $\forall^1$, $\exists^1$, and $\delta^1$}
\smallskip

We now recall three quantifiers which represent alternative definitions for the semantics of ordinary quantifiers $\exists$ and $\forall$. As we shall see, these alternative quantifiers do not preserve dimension in the same strong sense as the received $\exists$ and $\forall$. 

\begin{definition}If $a\in M$, let $F_a$ be the constant function $F_a(s)=\{a\}$ for all $s\in T$. The \emph{$\exists^1$}-quantifier is defined as follows:
$M\models_T\exists^1x\phi$ if for some $a\in M$ we have $M\models_{T[F_a/x]}\phi$.  The \emph{$\forall^1$}-quantifier is defined as follows:  
$M\models_T\forall^1x\phi$ if for all $a\in M$ we have $M\models_{T[F_a/x]}\phi.$
   The \emph{public announcement}-quantifier $\delta^1x$ is defined as follows:
$M\models_T\delta^1x\phi$ if for all $a\in M$ we have $M\models_{T_a}\phi$, where $T_a=\{s\in T:s(x)=a\}$.
\end{definition}

We shall now see that the quantifiers $\forall^1$, $\delta^1$ and $\exists^1$ do not preserve upper dimension, whence they are not Lindstr\"om quantifiers in the sense of Definition~\ref{Linquantsem}.

\begin{lemma}[\cite{MR3038038}]\label{forallone}
\hphantom{plaa}
\begin{enumerate}[(a)]
\item $\models\forall^1x\phi(x)\leftrightarrow\forall x(\dep(x)\to\phi(x))$
\item $\models\delta^1x\phi(x)\leftrightarrow\forall^1 y(x\ne y\vee\phi(x))$
\item $\models\dep(x_1,...,x_n,y)\leftrightarrow\delta^1x_1...\delta^1x_n\dep(y)$
\item $\models\dep(x_1,...,x_n,y)\leftrightarrow\forall^1z_1...\forall^1z_n(z_1\ne x_1\vee\ldots\vee z_n\ne x_k\vee\dep(y))$.
\end{enumerate}\end{lemma}

Items (a) and (b) show that $\forall^1 x\phi$ and $\delta^1x\phi(x)$
increase upper dimension of $\phi$ at most exponentially.  Items (c) and (d) 
shows that, as  operators,
$\delta^1x\phi(x,y)$ and $\forall^1x\phi(x)$ increase dimension in the
worst case exponentially. This shows that we cannot hope to prove that
they are in general better than exponential. This also shows that
these operators do not arise from a Lindstr\"om quantifier.

Note that by iterating $\forall^1x$ or $\delta^1x$ we can defined
dependence atoms of arbitrary arity. This shows that $\forall^1x$ and
$\delta^1x$ increase dimension more than any $k$-ary atom for a fixed
$k$.

\begin{lemma}
\hphantom{plaa}
\begin{enumerate}[(a)]
\item $\models\exists^1 x\phi\leftrightarrow\exists x(\dep( x)\wedge\phi)$. 
\item $\models \dep(x)\leftrightarrow\exists^1y(x=y)$.

\end{enumerate}
\end{lemma}

\begin{proof}
Easy.
\end{proof}
Hence $\exists^1$  increases upper dimension at most linearly.
Also, $\exists^1$ does indeed increase dimension, as the dimension of $x=y$ is 1 and the dimension of $\dep(x)$ is $n$. Hence $\exists^1$ is not first order definable and not definable even if we add arbitrary Lindstr\"om quantifiers to first order logic.

The point is that $\exists^1$ preserves dimension in the growth class where constancy logic  is, but not in the lower growth class where $\FO$ is.

\subsubsection*{Uniform definability}
\smallskip

Uniform definability, introduced by P. Galliani, is a phenomenon which does not exist in classical logic. It seems to be particularly characteristic to team based logics. Roughly speaking, a quantifier $Qx\phi(x,y)$ is uniformly definable in a logic if there is a single definition which works by substitution. In classical logic all definitions are uniform. In team based logics some quantifiers are definable but the definition is not uniform. In this section we use our dimension theory to prove this fact.

\begin{definition}[\cite{MR3038038}]
A generalized quantifier (which need not be a Lindstr\"om quantifier) $Q$ of a logic $L_1$ is said to be \emph{{uniformly definable}} in another logic $L_2$ if the logic $L_2$ has a sentence $\Phi(P)$, $P$ unary, with only positive occurrences of $P$,
 such that for all formulas $\phi(x,y)$ of the logic $L_1$ we have
$$\models Qx\phi(x,y)\leftrightarrow\Phi(\phi(z,y)/P(z)).$$ Similarly, if there are several formulas, as in $Qxy\phi(x,z)\psi(y,z)$.
\end{definition}

\begin{example}The equivalence
$$\models\exists^1x\phi(x,y)\leftrightarrow\exists x(\dep(x)\wedge\phi(x,y))$$ shows that the quantifier $\exists^1$ is uniformly definable in dependence logic, with $\Phi(P)$ the formula $\exists x(\dep(x)\wedge P(x))$.
The equivalence 
$$\models\phi\ \ilor\ \psi\leftrightarrow\exists x\exists y(\dep(x)\wedge\dep(y)\wedge ((x=y\wedge \phi)\vee(\neg x=y\wedge\psi)))$$ shows that the intuitionistic disjunction is uniformly definable in dependence logic, with $\Phi(P_0,P_1)$ the formula $\exists x\exists y(\dep(x)\wedge\dep(y)\wedge ((x=y\wedge P_0)\vee(\neg x=y\wedge P_1)$.

\end{example}

\begin{lemma}\label{uniform}
Suppose $$\models Qx\phi(x,y)\leftrightarrow\Phi(\phi(z,y)/P(z))$$ where $\Phi(P)$ is a sentence in dependence logic. 
Then $$\Dim_{Qx\phi(x,y),x y}(n)\le (n^{n^m}\cdot \Dim_{\phi(x,y)}(n))^k,$$ where $k$ is the length of $\Phi(P)$ and $m$ is the maximum of the lengths of $\vec{x}$ such that $\dep(\vec{x},y)$ for some  $y$ occurs in $\Phi(P)$.
\end{lemma}

\begin{proof}
We use induction on $\Phi$. The cases of atoms $\dep(\vec{x},y)$, the atom $P(z)$ and other atomic formulas are clear. The induction step for the connectives and the first order quantifiers follow from 
Corollary~\ref{DimPresCor}.
\end{proof}

\begin{corollary}[\cite{MR3038038}]\label{corforallone}
The quantifier $\forall^1$ is not uniformly definable in dependence logic.
\end{corollary}

\begin{proof} Suppose $\Phi(P)$, a sentence of length $l$, defines  $\forall^1$ uniformly in dependence logic. Let $m$ be as in Lemma~\ref{uniform}. Then there is by Lemma~\ref{forallone} a  formula $\Psi(P)$ of dependence logic, obtained from $\Phi(P)$ by $k$ repeated substitutions, which defines  $\dep(x_1,\ldots,x_k,y)$. By Lemma~\ref{uniform} we obtain an upper bound of $n^{n^m\cdot l^k}$ for $\Dim_{\dep(x_1,\ldots,x_k,y),\vec{x}y}(n)$.  However, we know from Example~\ref{atomiendimensioita} that
$\Dim_{\dep(x_1,\ldots,x_k,y),\vec{x}y}(n)=n^{n^k}$.
\end{proof}

Although Corollary \ref{corforallone} is not new, its proof shows that the concept of upper dimension offers a general method for demonstrating failure of uniform definability.


\subsubsection*{The ``at most half" atom}
\smallskip

\begin{definition}[The ``at most half" atom]Suppose $\len(\vx)=k$ and the model $M$ has size $n$. We define a new atom as follows:
$M\models_T\  {\tt\bf H}(\vx)$ if $|\{s(\vx) \mid s\in T\}|\le n^k/2$.

\end{definition}

Note that ${\tt\bf H}(\vx)$ is clearly definable in dependence logic (see Example~\ref{halfandparity}).

\begin{theorem}
Suppose $\len(\vec{x})=k$. The upper dimension of  ${\tt\bf H}(\vx)$ is $\sim \sqrt{\frac{2}{\pi}}2^{n^k-\frac{k}{2}log(n)}$.
\end{theorem}

\begin{proof}\cite[Page 4]{B1}

\end{proof}

\begin{corollary}Suppose $\len(\vx)=k$.
The atom ${\tt\bf H}(\vx)$ is not definable in the extension of first order logic by $<k$-ary dependence (or other) atoms.
\end{corollary}


\subsubsection*{The parity atom}
\smallskip

\begin{definition}[The parity atom]Suppose $\len(\vx)=k$. The $k$-ary \emph{parity atom} is defined by
$M\models_T\ev(\vx)$ if and only if $|\{s(\vx) \mid \vx\in T\}|$ is even.
\end{definition}

Note that $\ev(\vx)$ is   definable in independence logic (see Example~\ref{halfandparity}).

\begin{lemma}
The upper, dual and cylindrical dimension of the $k$-ary $\ev(\vx)$ is $2^{n^k-1}$.
\end{lemma}

\begin{proof}
This is a special case of 
Example~\ref{basic:evenEx}.  
\end{proof}

\begin{corollary} The $k$-ary
$\ev(\vx)$ is definable from the independence atoms but not from $l$-ary independence atoms for $l<k$.
\end{corollary}


\subsubsection*{(In)dependence friendly logic}
\smallskip

The so-called \emph{dependence friendly existential quantifier}, as in
independence friendly logic (\cite{MSS}), can be defined in terms of the
dependence atom. Hence we can estimate its effect on the dimension of
a formula.  We have
$$\begin{array}{lcl}
 \models\quad\exists x/\vy\phi&\leftrightarrow&\exists x(\dep(\vy,x)\wedge\phi)       \\
 \models\quad \dep(\vy,x)&\leftrightarrow&\exists z/\vy(z=x)      
\end{array}$$

\begin{corollary}
The quantifier $\exists x/\vy$, $\len(\vy)=k$, is not definable in the extension of first order logic by $<k$-ary independence, inclusion, exclusion, dependence  and constancy atoms.
\end{corollary}

A kind of ``dependence friendly" disjunction can be defined as follows: 
$M\models_T\phi\vee_{\vec{x}}\psi$ if $T=Y\cup Z$ such that $M\models_Y\phi$, $M\models_Z\psi$ and if $s,s'\in T$ with $s(\vec{x})=s'(\vx)$, then $(s\in Y\Leftrightarrow s'\in Y)$ and $(s\in Z\Leftrightarrow s'\in Z)$.

\begin{lemma}\label{depfridis}
$\models\phi\vee_{\vx}\psi\leftrightarrow\exists u\exists v(\dep(\vx,u)\wedge\dep(\vx,v)\wedge(u=v\to\phi)\wedge(u\ne v\to\psi))$.
\end{lemma}

In the proof of Lemma~\ref{depfridis} it is actually enough to use the 2-valued dependence atom $\dep(\vx,y)\wedge(\dep(y)\vee\dep(y))$. This has dimension $2^{m^k}$, when $\len(\vx)=k$ and the domain has cardinality m. Dimension analysis shows the full dependence atom cannot be defined from the $s$-valued dependence atom $\dep_s(\vx,y)$, defined by
$$\dep(\vx,y)\wedge(\dep(y)\vee\ldots\vee\dep(y)) \mbox{\hspace{1cm}($s$ disjuncts)}$$ for any $s>0$.
The 2-valued dependence atom $=_2\!(\vx,y)$ can be defined from $\vee_{\vx}$ and constancy atoms as follows:
$$\exists u\exists v(\dep(u)\wedge \dep(v)\wedge(y=u\vee_{\vx}y=v)).$$
This shows that the operation $\phi\vee_{\vx}\psi$ does not preserve dimension. The situation is similar to the dependence friendly existential quantifier.

\section{VC-dimension}\label{vc}

An important dimension in finite combinatorics is the
Vapnik-Cervonenkis (VC) dimension of a family of sets. It is defined
as follows: Let us say that a set $A$ is \emph{shattered} by a family
$H$ of subsets of a finite set if $\{h\cap A\mid h\in H\}$ contains
all the subsets of $A$. The VC-dimension of $H$ is the largest
cardinality of a set shattered by $H$. This dimension has turned out
to be useful e.g. in learning theory (\cite{Vap}). However, it does not
have the same flexibility as our dimension concepts and does not seem
to be applicable in the kind of analysis we have at hand in this
paper.

The VC-dimension of the family of teams of an even number of
$k$-tuples in a domain of $n$ elements is ${n^k}$. Yet evenness can be
expressed in independence logic. As the VC-dimension of the
independence atom is 1, this shows that our logical operations do not
preserve VC-dimension.


\section{Cylindrical dimension and the DNF}\label{dnf}

Our cylindircal dimension for a family of sets is actually known in the study of disjunctive normal forms (DNF) of Boolean functions: 
%
Suppose $X=\{a_1,\ldots,a_n\}$ is a finite set. We fix a proposition symbol $p_i$ for each $i\in [1,n]$. Now subsets $A$ of $X$ correspond canonically to valuations (truth functions) $v_A$ of $\{p_1,\ldots,p_n\}$. Respectively, families $\mathcal{A}$ of subsets of $X$ correspond to Boolean functions on $\{p_1,\ldots,p_n\}$ and thereby to propositional formulas $\phi_{{\mathcal{A}}}$ in $\{p_1,\ldots,p_n\}$. This brings a connection between families of sets and Boolean functions (\cite{Od}). An interval $I=\{Y\subseteq X : A\subseteq Y\subseteq B\}$ corresponds
to the set $\bar{I}$ of valuations in which some proposition symbols have a fixed value, namely $p_i$ for $a_i\in A$ must be $1$ and $p_i$ for $a_i\notin B$ must be $0$. The set $\bar{I}$ can be defined in propositional logic with a conjunction of literals i.e. propositional symbols and their negations. If a family $\mathcal{A}$ of subsets of $X$ can be expressed as the union of $d$ intervals, then the defining formula $\phi_{{\mathcal{A}}}$ can be taken to be a disjunction of $d$ conjunctions of literals.
In the theory of Boolean functions our concept of cylindrical dimension corresponds exactly to the concept of length $m(f)$ of the shortest disjunctive normal form for the Boolean function $f$, meaning the smallest number of disjuncts in the disjunctive normal form of $f$. The conjunctions in such a ``minimal DNF" (where we also stipulate that these consist of as few variables as possible) are the well-known prime implicants of $f$. The algorithm of  \cite{Qu} and McCluskey determines these and hence also the number $m(f)$. 
A classic result about $m(f)$ is the following estimate (\cite{Gl}) for almost all $f$ of $n$ Boolean variables:
$$c_1\frac{2^n}{(\log n)\log\log n} < m(f) < c_2\frac{(\log\log n)2^n}{\log n}.$$
Thus this is also an estimate for the cylindrical dimension of almost all families of subsets of a set of $n$ elements.
The DNF-dimension has been studied extensively and more estimates have been found, see 
\cite{Ko,Ma,We,Ku,As,Rm}.
For example, 
in \cite{Ku} the following better lower bound is proved
\begin{equation}\label{kuz}
    \frac{(1-\epsilon_n)\cdot 2^n}{\log n-\log\log n},
    \end{equation}
    where $\lim \epsilon_n=0$,
for almost all Boolean functions on $n$ variables.

In the following application of the estimate (\ref{kuz}), we measure probabilities of team properties by using the uniform distribution for teams on $k+1$ variables in a model of size $n$. Note that in a non-rigid model a random team property is almost surely not definable in any logic. Therefore the interesting case is the definability of random team properties in rigid models.

\begin{corollary}
In the class of finite rigid models a random $k+1$-ary team property $(k\ge 1)$ is almost surely not definable in the extension of first order logic by $k$-ary dependence, independence, inclusion, exclusion and anonymity atoms. 
\end{corollary}

\begin{proof}
If a random $k+1$-ary team property is definable in $\LF_{k}$, its cylindrical dimension is asymptotically $n^{n^{k}}$. But by (\ref{kuz}) the cylindrical dimension is asymptotically almost surely at least of the order $2^{n^{k+1}}$.
\end{proof}
We do not know whether upper dimension and dual upper dimension have been isolated in the study of Boolean functions and whether they have a role there.

\section{Infinite models}

Our dimension analysis can be adapted to the realm of infinite domains but it does not have similar power. The infinite dimensions tend to be all the same and we do not get applications to definability. In fact, the hierarchy results are false in the following sense: Three and higher arity dependence atoms can  be expressed in terms of binary dependence atoms. The trick is to use the binary dependence atom to introduce a pairing function:

\begin{theorem}In infinite domains all dependence atoms are definable in terms of $2$-ary dependence atoms. Respectively, in infinite domains  the ternary independence atom $xyz\perp uvw$ can express all dependence, independence, inclusion,
    anonymity, and exclusion atoms.
\end{theorem}
\def\la{\langle}
\def\ra{\rangle}

\begin{proof}Suppose $(x,y)\mapsto\la x,y\ra$ is a pairing function (i.e. $\la x,y\ra =\la x',y'\ra$ if and only if $x=x'$ and $y=y'$) on the (infinite) domain.
We prove the following typical case:
\begin{equation}\label{pairing}
\begin{array}{lcl}
 \models\ \dep(xyz,u)&\leftrightarrow&
\forall x_1\forall y_1\exists u_1 (\dep(x_1y_1,u_1)\wedge\\
&&\forall x_2\forall y_2\exists u_2(\dep(x_2y_2,u_2)\wedge \\
&&((x_1=x_2\wedge y_1=y_2)\leftrightarrow u_1=u_2)\wedge\\
&&((x_1=x\wedge y_1=y\wedge x_2=u_1 \wedge y_2=z)\to \dep(u_2,u))))
\end{array}
\end{equation}

Suppose a team $T$ satisfies $\dep(xyz,u)$. Let $Y$ be the extension of $T$ by giving all possible values for 
$x_1,x_2,y_1$ and $y_2$. We further extend $Y$ to $Z$ by giving values to $u_1$ and $u_2$ as follows:
$$s(u_1)=\la s(x_1),s(y_1)\ra, s(u_2)=\la s(x_2),s(y_2)\ra.$$ Clearly, $Z\models\ \dep(x_1y_1,u_1)$ and 
$Z\models\ \dep(x_2y_2,u_2)$. Also, obviously, $Z\models (x_1=x_2\wedge y_1=y_2)\Leftrightarrow u_1=u_2$.
Suppose then $\{s,s'\}\subseteq Z$ satisfies $x_1=x\wedge y_1=y\wedge x_2=u_1 \wedge y_2=z$ and, moreover, $s(u_2)=s'(u_2)$. A direct calculation yields $s(u)=s'(u)$.

Conversely, suppose $T$ satisfies the right hand side of (\ref{pairing}). Thus, if  $T$ is extended by giving all possible values for 
$x_1,x_2,y_1$ and $y_2$, and then further extended to $Z$ by giving suitable values to $u_1$ and $u_2$, then
 $Z$ satisfies the quantifier-free part of the right hand side (\ref{pairing}). To prove the left hand side of (\ref{pairing}), suppose $s,s'\in T$ agree about $xyz$. Let $f$ be a function such that if $s\in Z$, then $s(u_1)=f(s(x_1),s(y_1))$. Then, if $s\in Z$, then $s(u_2)=f(s(x_2),s(y_2))$. Clearly, $f$ is one-one. A calculation yields $s(u_2)=s'(u_2)$. Since $Z$ satisfies $\dep(u_2,u)$, we obtain $s(u)=s'(u)$. \end{proof}

It remains open, whether the unary dependence atom or the binary independence atom have similar universal power.
It remains also open whether the arity hierarchy of the inclusion atom collapses.

\section{Conclusion}

We have defined three dimension like notions in discrete mathematics and applied them to obtain hierarchy and undefinability results in the area of team semantics. Our results demonstrate that in finite models the arity of atoms puts a definitive bound on what can be expressed. In terms of our approach, the arity of the atoms of a sentence  completely determines the dimension of the sentence, and team properties of higher dimension cannot be expressed even if we add all possible Lindstr\"om quantifiers. On the other hand, this is only true if certain nicely behaving logical operations are the only ones that are used. If certain strong (from the perspective of our approach) logical operations, such as the intuitionistic implication, are allowed, the dimension analysis fails. Thus our quantitative analysis can be used to show the rationale 
of choosing some logical operations over some others.

We list below some open questions that remain unanswered by our results: 

\begin{enumerate}[(1)]

    \item Is the $k$-ary dependence atom  definable in terms of  $k$-ary independence, exclusion, inclusion, anonymity, constancy atoms, and some Lindstr\"om quantifiers?
\item Is the $k$-ary anonymity atom  definable in terms of the
  $k$-ary inclusion atom?
\item Is the $k$-ary independence atom definable in terms of the $k$-ary pure independence atom?  \item Is the $(k,l,m)$-ary independence atom  definable in terms of the
  $\max(k,l)+m$-ary dependence,  anonymity,  exclusion
   and  inclusion atoms?

    \item Dependence, exclusion, inclusion, anonymity and independence atoms arise in a natural way from the classes $\cF$, $\cX$, $\cI_\osaj$, $\cY$ and $\cI_\perp$, and for each of these atoms we have proved an arity hierarchy result. Furtheoremore, all the classes are first-order definable.
    Does there exist some other first-order definable families $\cA\subseteq\{R\mid R\subseteq X_1\times\cdots\times X_n\}$ such that the corresponding 
    atoms satisfy similar hierarchy result, and first-order logic extended with the atoms is strictly contained in dependence/exclusion or inclusion logic? \item Our dimension functions are either polynomial or exponential. Is this a general phenomenon for first order definable atoms i.e. is there a Dichotomy Theorem for first order definable atoms? Is it a decidable question to decide whether the dimension function is polynomial?
\end{enumerate}

\bibliographystyle{unsrtnat}
\bibliography{tiki}

\end{document}